\documentclass[10pt]{amsart}
\usepackage{amssymb,amsmath,amsthm,amsfonts}
\usepackage{graphicx,mathtext, color,wrapfig,ytableau,mathtools}
\usepackage[foot]{amsaddr}
\usepackage{stmaryrd}
\usepackage{tikz,caption,enumerate}

\allowdisplaybreaks

\def\be{\begin{eqnarray}}
\def\ee{\end{eqnarray}}

\newtheorem{theorem}{Theorem}
\newtheorem{definition}[theorem]{Definition}
\newtheorem{lemma}[theorem]{Lemma}
\newtheorem{proposition}[theorem]{Proposition}

\newtheorem{corollary}[theorem]{Corollary}
\newtheorem{conjecture}[theorem]{Conjecture}
\newtheorem{example}[theorem]{Example}

\newtheorem{innercustomprop}{Proposition}
\newenvironment{customprop}[1]
  {\renewcommand\theinnercustomprop{#1}\innercustomprop}
  {\endinnercustomprop}

\definecolor{A-Cycle}{RGB}{255,69,0}
\definecolor{B-Cycle}{RGB}{0,0,255}
\definecolor{DehnCycle}{rgb}{0.2,0.6,0.2}

\title{Cauchy identities for genus 2 Schur polynomials}

\author{S.~Arthamonov}
\email[S.A.]{arthamonov@bimsa.cn}

\author{Sh.~Shakirov}
\email[Sh.Sh.]{shakirov@bimsa.cn}

\address[S.A.,Sh.Sh.]{Beijing Institute of Mathematical Sciences and Applications}

\author{W.~Yan}
\email[W.Y.]{wbyan@tsinghua.edu.cn}
\address[W.Y.]{Yau Mathematical Sciences Center, Tsinghua University}

\begin{document}

\ytableausetup{boxsize=5pt}

\begin{abstract}
Genus 2 Macdonald polynomials $\Psi^{(q,t)}_{j_1,j_2,j_3}$ generalize ordinary Macdonald polynomials in several aspects. First, they provide common eigenfunctions for commuting difference operators that generalize the Macdonald difference operators of type $A_1$. Second, the algebra generated by these difference operators together with multiplication operators admits an action of genus 2 mapping class group (MCG) that generalizes the well-known action of $SL(2,{\mathbb Z})$ for ordinary Macdonald polynomials. In this paper, one more important aspect of Macdonald theory is considered: the Cauchy identities. We construct a genus 2 generalization of Cauchy identities in the particular case when $t=q=1$, i.e. for genus 2 Schur polynomials.
\end{abstract}

\maketitle

\vspace{-5ex}
\section{Introduction}

Genus two Macdonald polynomials $\Psi_{j_1,j_2,j_3}^{(q,t)}$ for root system $A_1$ were introduced in \cite{ArthamonovShakirov'2019} as a unique solution to a system of recursive relations generalizing Pieri rule. They are labeled by triples of nonnegative integers $(j_1,j_2,j_3)$ satisfying the following admissibility condition.

\begin{definition}
We call a triple of nonnegative integers $(j_1,j_2,j_3)$ admissible if
\begin{equation}
|j_1 - j_2| \leq j_3 \leq j_1 + j_2\qquad \textrm{and}\qquad j_1 + j_2 + j_3\equiv 0\bmod 2.
\label{eq:AdmissibilityDefinition}
\end{equation}
Hereinafter, we denote by $\mathbf J$ the set of all admissible triples and refer to (\ref{eq:AdmissibilityDefinition}) as triangle inequality.
\label{def:AdmissibleTriple}
\end{definition}

\noindent
\begin{minipage}{\linewidth}
\begin{minipage}{0.5\linewidth}
\hspace{\baselineskip}In this paper we study the following limit of genus two Macdonald polynomials
\begin{equation}
\Phi_{j_1,j_2,j_3}:=\frac{\lim_{q\rightarrow 1}\Psi_{j_1,j_2,j_3}^{(t=q)}}{(j_1+1)(j_2+1)(j_3+1)},\quad (j_1,j_2,j_3)\in\mathbf J.
\label{eq:PhiLimitDefinition}
\end{equation}
We show that this limit exists and gives rise to a family of Laurent polynomials in three variables $x_{12},x_{13},x_{23}$ which we refer to as \textit{Genus two Schur polynomials}.
\end{minipage}
\hfill
\begin{minipage}{0.45\linewidth}
\centering
\begin{tikzpicture}[scale=0.75]
\draw[ultra thick] (-0.1,1.5) to[out=0,in=180] (1.5,1.1) to[out=0,in=180] (3.1,1.5) to[out=0,in=90] (4.65,0) to[out=270,in=0] (3.1,-1.5) to[out=180,in=0] (1.5,-1.1) to[out=180,in=0] (-0.1,-1.5) to[out=180,in=270] (-1.55,0) to[out=90,in=180] (-0.1,1.5);
\draw[ultra thick] (0.2-0.5,-0.1) to[out=50,in=130] (0.2+0.5,-0.1);
\draw[ultra thick] (0.2-0.7,0.1) to[out=-60,in=240] (0.2+0.7,0.1);
\draw[ultra thick] (2.8-0.5,-0.1) to[out=50,in=130] (2.8+0.5,-0.1);
\draw[ultra thick] (2.8-0.7,0.1) to[out=-60,in=240] (2.8+0.7,0.1);
\draw[color=B-Cycle,thick] (1.5,-0.7) to (1.5,0.7);
\draw[color=B-Cycle,thick] (1.5,0.7) to (3.5,0.7) to[out=0,in=90] (4,0.2) to[out=270,in=90] (4,-0.2) to[out=270,in=0] (3.5,-0.7) to (1.5,-0.7);
\draw[color=B-Cycle,thick] (1.5,-0.7) to (-0.5,-0.7) to[out=180,in=270] (-1,-0.2) to (-1,0.2) to[out=90,in=180] (-0.5,0.7) to (1.5,0.7);
\fill[color=B-Cycle] (1.5,0.7) circle (0.09);
\fill[color=white] (1.5,0.7) circle (0.05);
\fill[color=B-Cycle] (1.5,-0.7) circle (0.09);
\fill[color=white] (1.5,-0.7) circle (0.05);
\draw (-1.2,0) node[color=B-Cycle] {$j_1$};
\draw (1.3,0) node[color=B-Cycle] {$j_2$};
\draw (3.8,0) node[color=B-Cycle] {$j_3$};
\end{tikzpicture}

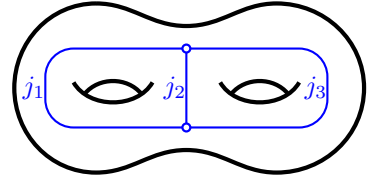
\captionof{figure}{Admissible triples.}
\end{minipage}
\end{minipage}
\smallskip

Here we aim to generalize to genus two Schur polynomials an important part of the theory of ordinary Schur polynomials: the Cauchy identities. For example, ordinary Schur polynomials of type $A_1$ have a form

\begin{equation}
S_{j}(x) = \dfrac{x^{j+1} - x^{-j-1}}{x - x^{-1}}
\end{equation}
and the following Cauchy identities hold:

\begin{equation}
\sum\limits_{j = 0}^{\infty} S_{j}(x) S_{j}(y) \lambda^{j} = \dfrac{1}{(1 - \lambda xy)(1 - \lambda x y^{-1})(1 - \lambda x^{-1} y)(1 - \lambda x^{-1} y^{-1})}
\end{equation}
\begin{equation}
\sum\limits_{j = 0}^{\infty} S_{j}(x) \lambda^{j} = \dfrac{1}{(1 - \lambda x)(1 - \lambda x^{-1})}
\end{equation}
Similarly, for genus two Schur polynomials, we consider infinite sums where the summand is a product of one, two, or more polynomials $\Phi_{j_1,j_2,j_3}$. Our main result Theorem \ref{th:Arctangent} is a proof of Cauchy sum where the summand is linear in the polynomials $\Phi_{j_1,j_2,j_3}$, for which we present a formula in terms of elementary functions (arctangent). For the more complicated cases which are quadratic or higher degree in the polynomials $\Phi_{j_1,j_2,j_3}$, we conjecture (see Conjecture \ref{conj:MainConjecture}) formulas in terms of generalized hypergeometric functions studied by Gelfand, Kapranov, Zelevinsky \cite{GelfandKapranovZelevinsky'1994}.

\section{Genus two Schur polynomials}

By virtue of their definition, genus two Schur polynomials inherit a number of major properties from  genus two Macdonald polynomials. In particular, $\Phi_{j_1,j_2,j_3}$ form a basis of common eigenfunctions of three second order differential operators and satisfy their own analog of Pieri rule. For the purpose of the proofs, however, it is more convenient for us to start by examining the limit of Pieri recursion rule first which would then be used to show that limit (\ref{eq:PhiLimitDefinition}) exists for all admissible triples $(j_1,j_2,j_3)\in\mathbf J$.

Let $\mathcal H$ be a $\mathbb Z_2\times\mathbb Z_2\times\mathbb Z_2$-invariant subring of Laurent polynomials in three variables
\begin{equation}
\mathcal H=\mathbb C[x_{12}^{\pm1},x_{13}^{\pm},x_{23}^{\pm}]^{\mathbb Z_2\times\mathbb Z_2\times\mathbb Z_2}=\mathbb C[x_{12}+x_{12}^{-1},x_{13}+x_{13}^{-1},x_{23}+x_{23}^{-1}].
\label{eq:HDefinition}
\end{equation}

\begin{definition}
We say that a collection of Laurent polynomials in $\mathcal H$ labeled by admissible triples of indices
\begin{equation*}
\phi_{j_1,j_2,j_3}\in\mathcal H,\qquad (j_1,j_2,j_3)\in\mathbf J
\end{equation*}
satisfies Pieri rule for genus two Schur polynomials if
\begin{subequations}
\begin{align}
\label{eq:PieriForSchur12}(x_{12} + x_{12}^{-1}) \ \phi_{j_1,j_2,j_3} = \sum_{a,b \in \{-1,+1\}} \ K_{a,b}(j_1,j_2,j_3) \ \phi_{j_1+a,j_2+b,j_3} \\[7pt]
\label{eq:PieriForSchur13}(x_{13} + x_{13}^{-1}) \ \phi_{j_1,j_2,j_3} = \sum_{a,b \in \{-1,+1\}} \ K_{a,b}(j_1,j_3,j_2) \ \phi_{j_1+a,j_2,j_3+b} \\[7pt]
\label{eq:PieriForSchur23}(x_{23} + x_{23}^{-1}) \ \phi_{j_1,j_2,j_3} = \sum_{a,b \in \{-1,+1\}} \ K_{a,b}(j_2,j_3,j_1) \ \phi_{j_1,j_2+a,j_3+b}
\end{align}
\label{eq:PieriForSchur}
\end{subequations}
for all $(j_1,j_2,j_3)\in\mathbf J$. Here on the right hand side we assume $\phi_{j_1,j_2,j_3}=0$ for non-admissible triples and 
\begin{equation}
K_{a,b}(j_1,j_2,j_3)=ab\,\frac{(aj_1+bj_2+j_3+a+b+2)(aj_1+bj_2-j_3+a+b)}{4 (j_1+1) (j_2+1)}.
\label{eq:Kab}
\end{equation}
\label{def:GenusTwoSchurPieri}
\end{definition}

\begin{lemma}
For a given admissible triple $(j_1,j_2,j_3)\in\mathbf J$ and $a,b=\pm1$, the coefficient $K_{a,b}(j_1,j_2,j_3)$ is non-vanishing if and only if the triple $(j_1+a,j_2+b,j_3)$ is admissible.
\label{lemm:NonvanishingKab}
\end{lemma}
\begin{proof}
First note that in all cases $a+b\equiv0\bmod 2$, so the parity condition in (\ref{eq:AdmissibilityDefinition}) is always satisfied for the triple $(j_1+a,j_2+b,j_3)$. Now for the triangle inequality we have to consider the following cases
\begin{itemize}
\item $(a,b)=(+1,+1)$. In this case the triple $(j_1+1,j_2+1,j_3)$ is always admissible and, at the same time, the coefficient $K_{+1,+1}=\frac{(j_1+j_2+j_3+4)(j_1+j_2-j_3+2)}{4(j_1+1)(j_2+1)}$ is always nonzero.
\item $(a,b)=(+1,-1)$. In this case $j_1+j_2$ is preserved. The $|j_1-j_2|$ increases only when $j_1\geq j_2$. Under such assumption the triple $(j_1+1,j_2-1,j_3)$ is admissible unless $j_1-j_2=j_3$. The latter happens exactly when the second factor in the numerator vanishes.
\item $(a,b)=(-1,+1)$. In this case $j_1+j_2$ is also preserved. The $|j_1-j_2|$ increases only when $j_2\geq j_1$. Under such assumption the triple $(j_1-1,j_2+1,j_3)$ is admissible unless $j_2-j_1=j_3$. The latter happens exactly when the second factor in the numerator vanishes.
\item $(a,b)=(-1,-1)$. In this case the coefficient reads $K_{-1,-1}=\frac{(-j_1-j_2+j_3)(-j_1-j_2+j_3+2)}{4(j_1+1)(j_2+1)}$. Assuming the original triple $(j_1,j_2,j_3)$ was admissible, this coefficient vanishes only when $j_3=j_1+j_2$, this is precisely the case when the second inequality of (\ref{eq:AdmissibilityDefinition}) fails for the new triple $(j_1-1,j_2-1,j_3)$.
\end{itemize}
\end{proof}

The following technical lemma allows one to establish a relation between Definition \ref{def:GenusTwoSchurPieri} and the Pieri rule for genus two Macdonald polynomials from \cite{ArthamonovShakirov'2019}.
\begin{lemma}
Let $a,b\in\{\pm1\}$ and suppose that the triples $(j_1,j_2,j_3)$ and $(j_1+a,j_2+b,j_3)$ are both admissible, we have the following limiting relation
\begin{equation*}
K_{a,b}(j_1,j_2,j_3)=\frac{(j_1+a+1)(j_2+b+1)}{(j_1+1)(j_2+1)}\lim_{q\rightarrow 1}\left(C_{a,b}(j_1,j_2,j_3)\big|_{t=q}\right),
\end{equation*}
where
\begin{equation*}
C_{a,b}(j_1,j_2,j_3) = a \ b \ \frac{
\left[ \frac{a j_1 + b j_2 + j_3}{2}, \frac{a + b + 2}{2} \right]_{q,t}
\left[ \frac{a j_1 + b j_2 - j_3}{2}, \frac{a + b}{2} \right]_{q,t}
\Big[j_1 - 1, 2\Big]_{q,t}
\Big[j_2 - 1, 2\Big]_{q,t}
}
{
\Big[j_1, \frac{a+3}{2} \Big]_{q,t}\Big[j_1 - 1, \frac{a+3}{2}\Big]_{q,t}\Big[j_2, \frac{b+3}{2} \Big]_{q,t}\Big[j_2 - 1, \frac{b+3}{2}\Big]_{q,t}
},
\end{equation*}
\begin{equation*}
[n,m]_{q,t} = \dfrac{ q^{\frac{n}{2}} t^{\frac{m}{2}} - q^{-\frac{n}{2}} t^{-\frac{m}{2}} }{ q^{\frac{1}{2}} - q^{-\frac{1}{2}} }.
\end{equation*}
\label{lemm:KabIsALimitOfCab}
\end{lemma}
\begin{proof}
Consider the following exhaustive list of possible cases
\begin{itemize}
\item $(a,b)=(+1,+1)$. In this case we get
\begin{equation*}
\frac{(j_1+2)(j_2+1)}{(j_1+1,j_2+1)}\lim_{q\rightarrow1}C_{+1,+1}(j_1,j_2,j_3)\big|_{t=q}=\frac{(j_1+j_2-j_3+2)(j_1+j_2+j_3+4)}{4(j_1+1)(j_2+1)}=K_{+1,+1}(j_1,j_2,j_3).
\end{equation*}
\item $(a,b)=(+1,-1)$ and $j_2>0$. This gives
\begin{equation*}
\frac{(j_1+2)j_2}{(j_1+1,j_2+1)}\lim_{q\rightarrow1}C_{+1,-1}(j_1,j_2,j_3)\big|_{t=q}=-\frac{(j_1-j_2-j_3)(j_1-j_2+j_3+2)}{4(j_1+1)(j_2+1)}=K_{+1,-1}(j_1,j_2,j_3).
\end{equation*}
\item $(a,b)=(-1,+1)$ and $j_1>0$. This gives
\begin{equation*}
\frac{j_1(j_2+2)}{(j_1+1)(j_2+1)}\lim_{q\rightarrow1}C_{-1,+1}(j_1,j_2,j_3)\big|_{t=q}=-\frac{(j_1-j_2-j_3-2)(j_1-j_2+j_3)}{4(j_1+1)(j_2+1)}=K_{-1,+1}(j_1,j_2,j_3).
\end{equation*}
\item $(a,b)=(-1,-1)$ and $j_1,j_2>0$. Finally, in this case we get
\begin{equation*}
\frac{j_1j_2}{(j_1+1)(j_2+1)}\lim_{q\rightarrow1}C_{-1,-1}(j_1,j_2,j_3)\big|_{t=q}=\frac{(j_1+j_2-j_3)(j_1+j_2+j_3+1)}{4(j_1+1)(j_2+1)}=K_{-1,-1}(j_1,j_2,j_3).
\end{equation*}
\end{itemize}
\end{proof}

\begin{proposition}
There exists a unique solution to a system of recursive equations (\ref{eq:PieriForSchur}) with initial condition $\phi_{0,0,0}=1$ and $\phi_{j_1,j_2,j_3}=0$ for non-admissible triples. Moreover, this solution is given by the following limit of the genus two Macdonald polynomials:
\begin{equation}
\phi_{j_1,j_2,j_3}=\Phi_{j_1,j_2,j_3}=\frac{\lim_{q\rightarrow1}\Psi_{j_1,j_2,j_3}^{(t=q)}}{(j_1+1)(j_2+1)(j_3+1)}.
\label{eq:LimitSolvesGenusTwoSchurPieri}
\end{equation}
\label{prop:GenusTwoPieriForSchurSolution}
\end{proposition}
\begin{proof}
We will prove the statement of this proposition by induction in $j_1+j_2+j_3=N$. Assuming the initial conditions $\phi_{0,0,0}=1$, for every $N\in\mathbb Z_{\geq0}$ consider the following inductive statement
\begin{equation*}
\mathop{\mathrm{Pieri}}(N):=\left\{\qquad\mathrm{\parbox{0.45\linewidth}{\textit{There exists a unique finite collection\\[0.7em] $\{\phi_{j_1,j_2,j_3}\;:\; (j_1,j_2,j_3)\in\mathbf J,\quad j_1+j_2+j_3\leq N\}$\\[0.7em]
which simultaneously
\begin{enumerate}[i)]
\item Satisfies (\ref{eq:LimitSolvesGenusTwoSchurPieri}) for all $j_1+j_2+j_3\leq N$,
\item Satisfies (\ref{eq:PieriForSchur}) for all $j_1+j_2+j_3\leq N-2$.
\end{enumerate}
}}}\right\}
\end{equation*}

The base of induction is given by $\mathop{\mathrm{Pieri}}(0)$ which holds true for our choice of initial condition $\phi_{0,0,0}=1$. As for the step of induction, we assume that $\mathop{\mathrm{Pieri}}(N)$ holds for some $N\geq0$.

Let $(i_1,i_2,i_3)\in\mathbf J$ be an admissible triple with $i_1+i_2+i_3=N+2$, it follows from Definition \ref{def:AdmissibleTriple} that at least on the triples $(i_1-1,i_2-1,i_3),(i_1-1,i_2,i_3-1),(i_1,i_2-1,i_3-1)$ must be admissible. For each such admissible triple, equations (\ref{eq:PieriForSchur}) provide one linear equation on $\phi_{i_1,i_2,i_3}$. As a result we get at least one, but no more than three equations on $\phi_{i_1,i_2,i_3}$. These equations must be compatible by Proposition 4 of \cite{ArthamonovShakirov'2019} combined with Lemma \ref{lemm:KabIsALimitOfCab}. Next, using together Lemmas \ref{lemm:NonvanishingKab} and \ref{lemm:KabIsALimitOfCab} we conclude that the limit $\lim_{q\rightarrow1}\Psi_{i_1,i_2,i_3}^{(t=q)}$ exists and must be equal to $(j_1+1)(j_2+1)(j_3+1)\phi_{i_1,i_2,i_3}$.

Applying the above argument to all admissible triples $(i_1,i_2,i_3)\in\mathbf J$ at the level $i_1+i_2+i_3=N+2$ we uniquely extend the solution of (\ref{eq:PieriForSchur}) and conclude that $\mathop{\mathrm{Pieri}}(N+2)$ holds true.
\end{proof}

From Proposition \ref{prop:GenusTwoPieriForSchurSolution} we conclude that genus two Schur polynomials have the following symmetry with respect to simultaneous permutations of indexes and variables
\begin{equation}
\phi_{j_1,j_2,j_3}(x_{12},x_{13},x_{23})=\phi_{j_{\sigma_1},j_{\sigma_2},j_{\sigma_2}}(x_{\sigma_1\sigma_2},x_{\sigma_1\sigma_3},x_{\sigma_2\sigma_3}),\qquad \textrm{for all}\quad (j_1,j_2,j_3)\in\mathbf J\quad\quad \textrm{and}\quad \sigma\in S_3.
\label{eq:GenusTwoSchurS3Symmetry}
\end{equation}

\begin{lemma}
Genus two Schur polynomials are normalized to have unit value at the special point $x_{12}=x_{13}=x_{23}=1$:
\begin{equation*}
\phi_{j_1,j_2,j_3}(1,1,1)=1\qquad \textrm{for all}\quad (j_1,j_2,j_3)\in\mathbf J.
\end{equation*}
\label{lemm:GenusTwoSchurNormalizedToSpecialValue}
\end{lemma}
\begin{proof}
We will prove the statement of the Lemma by induction in $i_1+i_2+i_3=N$. The base case $N=0$ is given by initial condition. As for the step of induction we assume that the statement of the Lemma holds for all $n\leq N$. Let $(i_1,i_2,i_3)\in\mathbf J$ be an admissible triple with $i_1+i_2+i_3=N+2$, then at least one of the triples $(i_1-1,i_2-1,i_3), (i_1-1,i_2,i_3-1), (i_1,i_2-1,i_3-1)$ must be admissible. By (\ref{eq:GenusTwoSchurS3Symmetry}) we can assume without loss of generality that $(i_1-1,i_2-1,i_3)$ is admissible. Specializing (\ref{eq:PieriForSchur12}) to $x_{12}=x_{13}=x_{23}=1$ and $(j_1,j_2,j_3)=(i_1-1,i_2-1,i_3)$ we get necessary condition
\begin{equation}
\begin{split}
\phi_{i_1,i_2,i_3}(1,1,1)=&-\frac1{(i_1+i_2-i_3)(i_1+i_2+i_3+2)}\Big((i_1+i_2+i_3)(i_1+i_2+i_3-2)\phi_{i_1-2,i_2-2,i_3}(1,1,1)\\[0.3em]
&+(i_1-i_2+i_3)(i_2+i_3-i_1+2)\phi_{i_1-2,i_2,i_3}(1,1,1)-8 i_1 i_2\phi_{i_1-1,i_2-1,i_3}(1,1,1)\\[0.3em]
&-(i_1-i_2-i_3)(i_1+i_3-i_2+2)\phi_{i_1,i_2-2,i_3}(1,1,1)\Big).
\end{split}
\label{eq:phi1recursionproof}
\end{equation}

Note that coefficient in $\phi_{i_1-2,i_2-2,i_3}(1,1,1)$ vanishes precisely when the triple $(i_1-2,i_2-2,i_3)$ is nonadmissible. Same property holds true for the coefficients in $\phi_{i_1-2,i_2,i_3}(1,1,1)$ and $\phi_{i_1,i_2-2,i_3}(1,1,1)$. Combining the inductive assumption with these properties of the coefficients we conclude that without loss of generality one can always effectively substitute $\phi_{i_1-2,i_2-2,i_3}(1,1,1)=\phi_{i_1-2,i_2,i_3}(1,1,1)=\phi_{i_1-1,i_2-1}(1,1,1)=\phi_{i_1,i_2-2,i_3}(1,1,1)=1$ into (\ref{eq:phi1recursionproof}) regardless of whether the triples $(i_1-2,i_2-2,i_3), (i_1-2,i_2,i_3),$ and $(i_1,i_2-2,i_3)$ are actually admissible. This, after simplification, gives $\phi_{i_1,i_2,i_3}(1,1,1)=1$ which proves the step of induction.
\end{proof}

\begin{lemma}
For every admissible triple $(j_1,j_2,j_3)\in\mathbf J$ the top degree part of the genus two Schur polynomial $\phi_{j_1,j_2,j_3}$ has total degree $(j_1+j_2+j_3)/2$ and consists of a single monomial of the form
\begin{subequations}
\begin{equation}
c\, x_{12}^{d_3}x_{13}^{d_2}x_{23}^{d_1},
\end{equation}
where
\begin{equation}
d_{1} = \dfrac{-j_1+j_2+j_3}{2}, \qquad \ d_{2} = \dfrac{j_1-j_2+j_3}{2}, \qquad d_{3} = \dfrac{j_1+j_2-j_3}{2}.
\label{eq:ChangeBetweenJD}
\end{equation}
\end{subequations}
and $c\in\mathbb C\backslash\{0\}$ is a nonzero constant.
\label{lemm:LeadingTermSchur}
\end{lemma}
\begin{proof}
Use Lemma \ref{lemm:NonvanishingKab} and induction by $j_1+j_2+j_3=N$.
\end{proof}

\begin{proposition}
The collection of genus two Schur polynomials forms a $\mathbb C$-basis on $\mathcal H$
\begin{equation}
\left\{\phi_{j_1,j_2,j_3}\;|\; (j_1,j_2,j_3)\in\mathbf J\right\}.
\label{eq:SchurBasisH}
\end{equation}
\end{proposition}
\begin{proof}
By Lemma \ref{lemm:LeadingTermSchur} we know that all $\phi_{j_1,j_2,j_3}$ with $(j_1,j_2,j_3)\in\mathbf J$ have distinct top degree terms and must be linearly independent. It remains to show that together they span $\mathcal H$. Indeed, from (\ref{eq:PieriForSchur}) it follows that the span of (\ref{eq:SchurBasisH}) is closed under multiplication by $x_{ij}+x_{ij}^{-1},\;1\leq i<j\leq 3$, the generators of the ring $\mathcal H$ and also contains the unit $\phi_{0,0,0}=1$.
\end{proof}

Consider the following three differential operators acting on $\mathbb C(x_{12},x_{13},x_{23})$, the rational functions in three variables
\begin{subequations}
\begin{align}
\begin{split}
    \hat H_1:=&x_{12}^2 \frac{\partial^2}{\partial x_{12}^2}+x_{13}^2\frac{\partial^2}{\partial x_{13}^2}+\frac{2(x_{12}^2+1)(x_{13}^2+1) - 4x_{12}x_{13}(x_{23} + x_{23}^{-1} )}{(x_{12}-x_{12}^{-1})(x_{13}-x_{13}^{-1})} \frac{\partial^2}{\partial x_{12}\partial x_{13}}\\
    &\qquad+\frac{3x_{12}^2+1}{x_{12}-x_{12}^{-1}}\frac{\partial}{\partial x_{12}}+\frac{3x_{13}^2+1}{x_{13}-x_{13}^{-1}}\frac{\partial}{\partial x_{13}}+1,
\end{split}
\\[1em]
\begin{split}
    \hat H_2:=&x_{12}^2\frac{\partial^2}{\partial x_{12}^2}+x_{23}^2 \frac{\partial^2}{\partial x_{23}^2}+\frac{2(x_{12}^2+1)(x_{23}^2+1)-4x_{12}x_{23}(x_{13}+x_{13}^{-1})}{(x_{12}-x_{12}^{-1})(x_{23}-x_{23}^{-1})}\frac{\partial^2}{\partial x_{23}\partial x_{23}}\\
    &\qquad+\frac{3x_{12}^2+1}{x_{12}-x_{12}^{-1}}\frac{\partial}{\partial x_{12}}+\frac{3x_{23}^2+1}{x_{23}-x_{23}^{-1}}\frac{\partial}{\partial x_{23}}+1,
\end{split}
\\[1em]
\begin{split}
    \hat H_3:=&x_{13}^2\frac{\partial^2}{\partial x_{13}^2}+x_{23}^2\frac{\partial^2}{\partial x_{23}^2}+\frac{2(x_{13}^2+1)(x_{23}^2+1)-4x_{13}x_{23}(x_{12}+x_{12}^{-1})}{(x_{13}-x_{13}^{-1})(x_{23}-x_{23}^{-1})}\frac{\partial^2}{\partial x_{13}\partial x_{23}}\\
    &\qquad+\frac{3x_{13}^2+1}{x_{13}-x_{13}^{-1}}\frac{\partial}{\partial x_{13}}+\frac{3x_{23}^2+1}{x_{23}-x_{23}^{-1}}\frac{\partial}{\partial x_{23}}+1.
\end{split}
\end{align}
\label{eq:Hamiltonians}
\end{subequations}

\begin{proposition}
Genus two Schur polynomials are common eigenfunctions of differential operators (\ref{eq:Hamiltonians}):
\begin{equation}
\hat H_i\phi_{j_1,j_2,j_3}=(j_i+1)^2\phi_{j_1,j_2,j_3}\qquad\textrm{for all}\quad (j_1,j_2,j_3)\in\mathbf J,\quad 1\leq i\leq3.
\label{eq:GenusTwoSchurH1Eigenvectors}
\end{equation}
\label{prop:GenusTwoSchurHiEigenfunctions}
\end{proposition}
\begin{proof}
Combine formula (\ref{eq:LimitSolvesGenusTwoSchurPieri}) with Proposition 13 of \cite{ArthamonovShakirov'2019}.
\end{proof}

\subsection*{Series expansion of Genus two Schur polynomials.}

Below we investigate the the series expansion of Genus two Schur polynomials at $x_{12}=x_{13}=x_{23}=1$. For this purpose, consider the following substitution
\begin{equation}
x_{12}=1+X_{12},\qquad x_{13}=1+X_{13},\qquad x_{23}=1+X_{23}.
\label{eq:ChangeOfVariablesxToX}
\end{equation}
Let $\phi_{j_1,j_2,j_3}^{(m)}$ denote a homogeneous component of degree $m$ of the series expansion of the Genus two Schur polynomial:
\begin{equation*}
\phi_{j_1,j_2,j_3}=\sum_{m=0}^{+\infty}\phi_{j_1,j_2,j_3}^{(m)},\qquad \textrm{for all}\quad (j_1,j_2,j_3)\in\mathbf J.
\end{equation*}

Note that by Lemma \ref{lemm:GenusTwoSchurNormalizedToSpecialValue} that this series expansion has unit leading term $\phi_{j_1,j_2,j_3}^{(0)}=1$ for all $(j_1,j_2,j_3)\in\mathbf J$. Moreover, because $\phi_{j_1,j_2,j_3}$ is symmetric with respect to change $x_{i,j}\leftrightarrow x_{i,j}$ for all $1\leq i<j\leq3$, the homogeneous component of degree one must vanish $\phi_{j_1,j_2,j_3}^{(1)}=0$ for all $(j_1,j_2,j_3)\in\mathbf J$.

The coefficients of differential operators (\ref{eq:Hamiltonians}) are all rational functions in variables $x_{12},x_{13},x_{23}$. After the change of variables (\ref{eq:ChangeOfVariablesxToX}) we can equivalently view them as operators acting on formal power series in $X_{12},X_{13},X_{23}$ which take values in Laurent power series
\begin{equation}
\hat H_k:\mathbb C[[X_{12},X_{13},X_{23}]]\rightarrow\mathbb C(\!(X_{12},X_{13},X_{23})\!),\qquad \hat H_k=\sum_{m=-2}^{+\infty}\hat H_k^{(m)},\qquad 1\leq k\leq 3.
\label{eq:HkHomogeneousComponentDecompsition}
\end{equation}
where $\hat H_k^{(m)}$ stands for the homogeneous component of degree $m$. The leading coefficients of degree $m=-2$ read
\begin{subequations}
\begin{align}
\hat H_1^{(-2)}=&\frac{\partial^2}{\partial X_{12}^2}+\frac{\partial^2}{\partial X_{13}^2}+\frac{X_{12}^2+X_{13}^2-X_{23}^2}{X_{12}X_{13}}\frac{\partial^2}{\partial X_{12}\partial X_{13}}+\frac2{X_{12}}\frac{\partial}{\partial X_{12}}+\frac2{X_{13}}\frac{\partial}{\partial X_{13}},
\label{eq:H1m2}\\[0.3em]
\hat H_2^{(-2)}=&\frac{\partial^2}{\partial X_{12}^2}+\frac{\partial^2}{\partial X_{23}^2}+\frac{X_{12}^2+X_{23}^2-X_{13}^2}{X_{12}X_{23}}\frac{\partial^2}{\partial X_{12}\partial X_{23}}+\frac2{X_{12}}\frac{\partial}{\partial X_{12}}+\frac2{X_{23}}\frac{\partial}{\partial X_{23}},
\label{eq:H2m2}\\[0.3em]
\hat H_3^{(-2)}=&\frac{\partial^2}{\partial X_{13}^2}+\frac{\partial^2}{\partial X_{23}^2}+\frac{X_{13}^2+X_{23}^2-X_{12}^2}{X_{13}X_{23}}\frac{\partial^2}{\partial X_{13}\partial X_{23}}+\frac2{X_{13}}\frac{\partial}{\partial X_{13}}+\frac2{X_{23}}\frac{\partial}{\partial X_{23}}.
\label{eq:H3m2}
\end{align}
\label{eq:Hm2}
\end{subequations}

\begin{proposition}
Let $p\in\mathbb C[X_{12},X_{13},X_{23}]$ be a polynomial annihilated by all three of the operators (\ref{eq:Hm2}), then necessarily it is a constant polynomial $p\in\mathbb C$. In other words,
\begin{equation*}
\ker\hat H_1^{(-2)}\cap\ker\hat H_2^{(-2)}\cap\ker\hat H_3^{(-2)}=\mathbb C.
\end{equation*}
\label{prop:KernelHm2}
\end{proposition}

The proof of this proposition is intriguing in its own right but does not provide significant insight for the remainder of the manuscript. Consequently, we have omitted it from the main text and presented it in Appendix \ref{sec:ProofPropositionKernelHm2}.

Expanding both sides of (\ref{eq:LimitSolvesGenusTwoSchurPieri}) as a Laurent series in $X_{12},X_{13},X_{23}$ we get
\begin{equation}
\sum_{m=-2}^{+\infty}\hat H_k^{(m)}\phi_{j_1,j_2,j_3}^{(l-m)}=(j_k+1)^2\phi_{j_1,j_2,j_3}^{(l)} \qquad \textrm{for all}\quad k\in\{1,2,3\}.
\label{eq:DifferentialEquationsLaurentSeriesPsi}
\end{equation}
Note that for each integer $l\geq-2$ equation (\ref{eq:DifferentialEquationsLaurentSeriesPsi}) provides a system of three differential equations on $\phi_{j_1,j_2,j_3}^{(l+2)}$.

\begin{proposition}
There exists a unique solution of the system of differential equations (\ref{eq:DifferentialEquationsLaurentSeriesPsi}) with initial condition 
\begin{equation}
\phi_{j_1,j_2,j_3}^{(0)}=1,\qquad \phi_{j_1,j_2,j_3}^{(1)}=0\qquad\textrm{for all}\qquad (j_1,j_2,j_3)\in\mathbf J.
\label{eq:phiSeriesRecursionInitialCondition}
\end{equation}
Moreover, the solution has the following series expansion
\begin{subequations}
\begin{equation}
\phi_{j_1,j_2,j_3}=\sum_{m_{12},m_{13},m_{23}\in\mathbb Z_{\geq0}}c_{m_{12},m_{13},m_{23}}(j_1,j_2,j_3)X_{12}^{m_{12}}X_{13}^{m_{13}}X_{23}^{m_{23}},
\label{eq:phiSeriesDecompositionX}
\end{equation}
where coefficients are polynomials in variables $j_1,j_2,j_3$ with degrees bounded above as
\begin{equation}
c_{m_{12},m_{13},m_{23}}\in\mathbb C[j_1,j_2,j_3],\qquad \deg c_{m_{12},m_{13},m_{23}}\leq m_{12}+m_{13}+m_{23}.
\label{eq:DegreeBoundscm12m13m23}
\end{equation}
\end{subequations}
\end{proposition}
\begin{proof}
To prove the statement of the proposition we will proceed by induction in degree. Denote by $\mathcal S(M),\; M\in\mathbb Z_{\geq0}$ the following statement: ``\textit{There exists a unique family of polynomials 
$c_{m_{12},m_{13},m_{23}}\in\mathbb C[j_1,j_2,j_3],\; m_{12}+m_{13}+m_{23}\leq M$ satisfying (\ref{eq:DegreeBoundscm12m13m23}) which solve differential equation (\ref{eq:DifferentialEquationsLaurentSeriesPsi}) for all $0\leq l\leq M-1$ with initial conditions (\ref{eq:phiSeriesRecursionInitialCondition}). Moreover, $c_{m_{12},m_{13},m_{23}}(j_1,j_2,j_3)$ coincide with the coefficients of the series expansion of the Genus two Schur polynomials.}''

The base of induction $\mathcal S(0)$ is given by
\begin{equation*}
c_{0,0,0}=1,\qquad c_{1,0,0}=c_{0,1,0}=c_{0,0,1}=0.
\end{equation*}

To prove the step of induction, note that differential equation (\ref{eq:DifferentialEquationsLaurentSeriesPsi}) can be equivalently written as
\begin{equation}
\hat H_k^{(-2)}\phi^{(M+1)}_{j_1,j_2,j_3}=(j_k+1)^2\phi_{j_1,j_2,j_3}^{(M-2)}-\sum_{m=-1}^{+\infty}\hat H_k^{(m)}\phi_{j_1,j_2,j_3}^{(M-1-m)}\qquad\textrm{for all}\quad k\in\{1,2,3\}.
\label{eq:phiM1StepOfInduction}
\end{equation}
The latter can be viewed as a system of three differential equations on $\phi_{j_1,j_2,j_3}^{(M+1)}$. On the one hand, this system of differential equations admits at least one solution given by the series coefficient of the genus two Schur polynomials. On the other hand this solution must be unique by Proposition \ref{prop:KernelHm2}.

To conclude the proof of the step of induction we have to show that series coefficients $c_{m_{12},m_{13},m_{23}}(j_1,j_2,j_3),\;m_{12}+m_{13}+m_{23}=M+1$ are all polynomials in $j_1,j_2,j_3$ of degree $M+1$. Indeed, when (\ref{eq:phiM1StepOfInduction}) is expanded in monomial basis, it is equivalent to a system of nonhomogeneous linear equations on $c_{m_{12},m_{13},m_{23}}, m_{12}+m_{13}+m_{23}=M+1$ with constant coefficients. The right hand side is a vector of polynomials of degree $M+1$ in $j_1,j_2,j_3$. We already know that this system of linear equations has a unique solution, so the constant coefficient matrix must be invertible. As a corollary, all $c_{m_{12},m_{13},m_{23}}, m_{12}+m_{13}+m_{23}=M+1$ are polynomials of degree $M+1$.
\end{proof}

\begin{example}
To illustrate the polynomial dependence of the series coefficients $c_{m_{12},m_{13},m_{23}}$ on $j_1,j_2,j_3$ we list few nontrivial ones
\begin{equation*}
\begin{aligned}
c_{2,0,0}=&-c_{3,0,0}=\frac{j_1^2+j_2^2-j_3^2}{12}+\frac{j_1+j_2-j_3}{6},\\[0.5em]
c_{4,0,0}=&\frac{3j_1^4+2j_1^2j_2^2-6j_1^2j_3^2+3j_2^4-6j_2^2j_3^2+3j_3^4}{960}+\frac{j_1^3+2j_1^2j_2-3j_1^2j_3+j_1j_2^2-3j_1j_3^2+3j_2^2}{240}\\[0.5em]
&+\frac{-j_2^2j_3-j_2j_3^2+j_3^3}{80}+\frac{10j_1^2+j_1j_2-3j_1j_3+10j_2^2-3j_2j_3-7j_3^2}{120}+\frac{17(j_1+j_2-j_3)}{120},\\[0.5em]
c_{2,2,0}=&\frac{j_1^4+2j_1^2j_2^2+2j_1^2j_3^2-3j_2^4+6j_2^2j_3^2-3j_3^4}{480}+\frac{j_1^3+j_1^2j_2+j_1^2j_3+j_1j_2^2+j_1j_3^2-3j_2^3+3j_2^2j_3+j_2j_3^2-3j_3^3}{120}\\[0.5em]
&+\frac{j_1^2+2j_1j_2+2j_1j_3-3j_2^2+6j_2j_3-3j_3^2}{120}.
\end{aligned}
\end{equation*}
\end{example}

\section{Single Schur sum}

In this section, we derive the analogue of the Cauchy formula for genus two Schur polynomials. To achieve this, we introduce their weighted sum with two parameters, $\kappa$ and $\lambda$ to which we refer as \textit{Genus two Cauchy sum}. We demonstrate that this Cauchy sum satisfies a system of partial differential equations in the variables $\kappa$, $\lambda$, $x_{12}$, $x_{13}$, and $x_{23}$. Consequently, we derive a closed-form expression for the asymptotic behavior as $\lambda\rightarrow\kappa$ of the Genus two Cauchy sum in terms of elementary functions.

We begin with the following technical observation: 
\begin{lemma}
Fix a nonnegative integer $m \in \mathbb{Z}_{\geq 0}$. There exist only finitely many admissible triples $(j_1, j_2, j_3) \in \mathbf{J}$ satisfying $j_2 + j_3 = m$. 
\label{lemm:FinitelyManyj1}
\end{lemma}
\begin{proof}
Observe that $j_2, j_3 \leq m$. Using the triangle inequality (\ref{eq:AdmissibilityDefinition}), we get $j_1 \leq j_2 + j_3 = m$.
\end{proof}

\begin{definition}
Genus two Cauchy sum is defined by following formal power series in $\lambda$:
\begin{equation}
\mathcal{C}_-(x_{12}, x_{13}, x_{23} | \kappa, \lambda) := \sum_{(j_1, j_2, j_3) \in \mathbf{J}} \phi_{j_1, j_2, j_3} (\kappa^{j_1 + 1} - \kappa^{-j_1 - 1}) \lambda^{j_2 + j_3} \quad \in \quad \mathbb{C}[\kappa^{\pm 1}][x_{12}^{\pm 1}, x_{13}^{\pm 1}, x_{23}^{\pm 1}] \llbracket \lambda \rrbracket. 
\label{eq:CMinusDefinition}
\end{equation}
\end{definition}

Note that, by Lemma \ref{lemm:FinitelyManyj1}, each coefficient in $\lambda^{j_2 + j_3}$ is given by a finite sum. Consequently, the coefficients of the formal power series in $\lambda$ must be Laurent polynomials in the remaining variables $\kappa$, $x_{12}$, $x_{13}$, and $x_{23}$.

It will be convenient for us to introduce the following notation for the logarithmic derivative of the Genus two Cauchy sum
\begin{align}
\mathcal C_+(x_{12},x_{13},x_{23}|\kappa,\lambda):=&\kappa\frac{\partial}{\partial \kappa}\mathcal C_-(x_{12},x_{13},x_{23}|\kappa,\lambda)
=\sum_{(j_1,j_2,j_3)\in\mathbf J}\phi_{j_1,j_2,j_3}(\kappa^{j_1+1}+\kappa^{-j_1-1})(j_1+1)\lambda^{j_2+j_3}.
\label{eq:CPlusDefinition}
\end{align}

\begin{lemma}
Genus two Cauchy sum and its logarithmic derivative satisfy the following equation
\begin{equation}
\kappa \frac{\partial}{\partial \kappa}\mathcal C_+=\hat H_1\mathcal C_-.
\label{eq:H1CMinus}
\end{equation}
\end{lemma}
\begin{proof}
Applying $\hat H_1$ termwise to (\ref{eq:CMinusDefinition}) and using (\ref{eq:GenusTwoSchurH1Eigenvectors}) we get
\begin{equation}
\hat H_1\mathcal C_-(x_{12},x_{13},x_{23}|\kappa,\lambda)=\sum_{(j_1,j_2,j_3)\in\mathbf J}(j_1+1)^2\phi_{j_1,j_2,j_3}(\kappa^{j_1+1}-\kappa^{-j_1-1})\lambda^{j_2+j_3}=\kappa\frac{\partial}{\partial \kappa}\mathcal C_+.
\end{equation}
\end{proof}
\begin{corollary}
Genus two Cauchy sum and its logarithmic derivative both satisfy second order differential equation
\begin{equation}
\left(\kappa\frac{\partial}{\partial \kappa}\right)^2\mathcal C_{\pm}=\hat H_1\mathcal C_{\pm}
\label{eq:SecondOrderLogarithmicDerivativeCPlusMinus}
\end{equation}
\end{corollary}
\begin{proof}
Substituting definition (\ref{eq:CPlusDefinition}) of $\mathcal C_+$ into (\ref{eq:H1CMinus}) we conclude that (\ref{eq:SecondOrderLogarithmicDerivativeCPlusMinus}) holds for $\mathcal C_-$. On the other hand, using the fact that logarithmic derivative in $\kappa$ commutes with differential operator $\hat H_1$ we get
\begin{equation*}
\hat H_1\mathcal C_+\stackrel{(\ref{eq:CPlusDefinition})}{=}\hat H_1\left(\kappa\frac{\partial}{\partial\kappa}\right)\mathcal C_-=\left(\kappa\frac{\partial}{\partial\kappa}\right)\hat H_1\mathcal C_-\stackrel{(\ref{eq:H1CMinus})}{=}\left(\kappa\frac{\partial}{\partial\kappa}\right)^2\mathcal C_+.
\end{equation*}
\end{proof}

It will be convenient for us to consider a new set of variables
\begin{subequations}
\begin{equation}
(x_{12},x_{13},x_{23}|\kappa,\lambda)\leftrightarrow(\widetilde X_{12},\widetilde X_{13},\widetilde X_{23}|\widetilde\kappa,\widetilde \lambda),
\end{equation}
\begin{equation}
\kappa=\widetilde\kappa,\qquad \lambda=\widetilde \lambda,\qquad x_{ij}=1+\left(1-\widetilde\lambda/\widetilde \kappa\right)\widetilde X_{ij},\qquad 1\leq i<j\leq3.
\end{equation}
\label{eq:ChangeOfVariablesXTilde}
\end{subequations}

\begin{proposition}
In terms of variables (\ref{eq:ChangeOfVariablesXTilde}) the genus two Cauchy sum has the following series expansion as a formal power series in $\widetilde X_{12}, \widetilde X_{13},\widetilde X_{23}$
\begin{equation*}
\mathcal C_{\pm}\big(\widetilde X_{12},\widetilde X_{13},\widetilde X_{23}|\widetilde \kappa,\widetilde \lambda\big)=\sum_{m_{12},m_{13},m_{23}\in\mathbb Z_{\geq0}}C^{\pm}_{m_{12},m_{13},m_{23}}\big(\widetilde\kappa,\widetilde\lambda\big)\widetilde X_{12}^{m_{12}}\widetilde X_{13}^{m_{13}}\widetilde X_{23}^{m_{23}},
\end{equation*}
where the series coefficients are rational functions in variables $\widetilde\kappa,\widetilde\lambda$:
\begin{equation*}
C_{m_{12},m_{13},m_{23}}^{-}\big(\widetilde\kappa,\widetilde\lambda\big)\quad\in\quad \mathbb C\big(\widetilde\kappa,\widetilde\lambda\big)\qquad\textrm{for all}\qquad m_{12},m_{13},m_{23}\in\mathbb Z_{\geq0}.
\end{equation*}
Moreover, as a function of $\widetilde\lambda$ for fixed value of $\kappa\neq0,1$, 
\begin{enumerate}
\item $C_{m_{12},m_{13},m_{23}}^{-}\big(\widetilde \kappa,\widetilde \lambda\big)$ has a pole at $\widetilde\lambda=\widetilde\kappa$ of order at most $2$,
\item $C_{m_{12},m_{13},m_{23}}^{+}\big(\widetilde\kappa,\widetilde\lambda\big)$ has a pole at $\widetilde\lambda=\widetilde\kappa$ of order at most $3$.
\end{enumerate}
Note that we do not specify the order of other poles.
\label{prop:PoleAtKappaEqualsLambda}
\end{proposition}
\begin{proof}
Substituting (\ref{eq:phiSeriesDecompositionX}) into (\ref{eq:CMinusDefinition}) and changing variables as in (\ref{eq:ChangeOfVariablesXTilde}) we obtain the following presentation of the coefficient $C_{m_{12},m_{13},m_{23}}^-\big(\widetilde\kappa,\widetilde\lambda\big)$ as series in $\widetilde\lambda$
\begin{equation}
\begin{aligned}
C_{m_{12},m_{13},m_{23}}^-\big(\widetilde\kappa,\widetilde\lambda\big)=&\left(1-\widetilde\lambda/\widetilde\kappa\right)^{m_{12}+m_{13}+m_{23}}\sum_{(j_1,j_2,j_3)\in\mathbf J}c_{m_{12},m_{13},m_{23}}(j_1,j_2,j_3)\left(\widetilde\kappa^{j_1}+\widetilde\kappa^{-j_1}\right)\widetilde\lambda^{j_2+j_3}\\
=&\left(1-\widetilde\lambda/\widetilde\kappa\right)^{m_{12}+m_{13}+m_{23}}\sum_{j_2,j_3\in\mathbb Z_{\geq0}}\left(\sum_{j_1=|j_2-j_3|}^{j_2+j_3}c_{m_{12},m_{13},m_{23}}(j_1,j_2,j_3)\left(\widetilde \kappa^{j_1}+\widetilde\kappa^{-j_1}\right)\right)\widetilde\lambda^{j_2+j_3}.
\end{aligned}
\label{eq:CMinusXtildeExpansion}
\end{equation}
To prove the statement of the proposition it will be enough for us to show that the right hand side of (\ref{eq:CMinusXtildeExpansion}) is a power series expansion in $\widetilde\lambda$ of a rational function in $\widetilde\kappa,\widetilde\lambda$ with the order of the pole at $\widetilde\lambda=\widetilde\kappa$ at most $2$.

We start by proving an auxiliary identity for power series decomposition of a rational function in three variables $\Lambda_1,\Lambda_2,\Lambda_3$ as a power series in the last two, namely we have
\begin{equation}
\begin{aligned}
\sum\limits_{(j_1,j_2,j_3)\in\mathbf J} \Lambda_1^{j_1}\Lambda_2^{j_2}\Lambda_3^{j_3} =& \sum\limits_{j_2,j_3} \dfrac{\Lambda_1^{j_2+j_3+2}-\Lambda_1^{|j_2-j_3|}}{\Lambda_1^2-1} \Lambda_2^{j_2} \Lambda_3^{j_3} \\ 
=& \sum\limits_{j_2 \geq j_3} \dfrac{\Lambda_1^{j_2+j_3+2}-\Lambda_1^{j_2-j_3}}{\Lambda_1^2-1} \Lambda_2^{j_2} \Lambda_3^{j_3} + \sum\limits_{j_2 < j_3} \dfrac{\Lambda_1^{j_2+j_3+2}-\Lambda_1^{j_3-j_2}}{\Lambda_1^2-1} \Lambda_2^{j_2} \Lambda_3^{j_3} \\
=& \dfrac{1}{(1 - \Lambda_1\Lambda_3)(1 - \Lambda_2 \Lambda_3)(1 - \Lambda_1^2 \Lambda_2 \Lambda_3)} + \dfrac{1}{(1 - \Lambda_1\Lambda_2)(1 - \Lambda_2 \Lambda_3)(1 - \Lambda_1^2 \Lambda_2 \Lambda_3)} \\
=& \dfrac{1}{(1-\Lambda_1\Lambda_2)(1-\Lambda_1\Lambda_3)(1-\Lambda_2\Lambda_3)}
\end{aligned}
\label{eq:Lambda1Lambda2Lambda3Sum}
\end{equation}

Note that the right hand side of (\ref{eq:Lambda1Lambda2Lambda3Sum}) specialize well to a rational function in two variables when $\Lambda_1 = \widetilde\kappa^{\pm 1}, \Lambda_2 = \Lambda_3 = \widetilde\lambda$. In particular, it implies that that $C_{0,0,0}^{-}\big(\widetilde \kappa,\widetilde \lambda\big)$ is a rational function in $\widetilde\kappa,\widetilde\lambda$ that has a pole of order at most 2 at $\widetilde\lambda=\widetilde\kappa$.

To prove the statement of the proposition for general values of $m_{12},m_{13},m_{23}\in\mathbb Z_{\geq0}$ note that from (\ref{eq:DegreeBoundscm12m13m23}) we know that $c_{m_{12},m_{13},m_{23}}(j_1,j_2,j_3)$ is a polynomial of degree at most $m_{12}+m_{13}+m_{23}$ in variables $j_1,j_2,j_3$. It implies the following identity for a general sum
\begin{subequations}
\begin{equation}
    \sum_{(j_1,j_2,j_3)\in\mathbf J} c_{m_{12},m_{13},m_{23}}(j_1,j_2,j_3)\Lambda_1^{j_1}\Lambda_2^{j_2}\Lambda_3^{j_3}=\mathcal D_{m_{12},m_{13},m_{23}}\left(\frac1{(1-\Lambda_1\Lambda_2)(1-\Lambda_1\Lambda_3)(1-\Lambda_2\Lambda_3)}\right),
\label{eq:DLambda1Lambda2Lambda3Sum}
\end{equation}
where $\mathcal D_{m_{12},m_{13},m_{23}}$ is a differential operator of order at most $m_{12}+m_{13}+m_{23}:$
\begin{equation}
\mathcal D_{m_{12},m_{13},m_{23}}\in\mathbb C\left[\frac\partial{\partial\Lambda_1},\frac\partial{\partial\Lambda_2},\frac\partial{\partial\Lambda_3}\right],\qquad \deg\mathcal D_{m_{12},m_{13},m_{23}}\leq m_{12}+m_{13}+m_{23}.
\end{equation}
\end{subequations}
As a corollary, all higher $C_{m_{12},m_{13},m_{23}}^{-}\big(\widetilde \kappa,\widetilde \lambda\big)$ are obtained by first specializing the rational function on the right hand side of (\ref{eq:DLambda1Lambda2Lambda3Sum}) to $\Lambda_1=\kappa^{\pm1},\Lambda_2=\Lambda_3=\lambda$ followed by multiplication by $(1-\widetilde\lambda/\widetilde\kappa)^{m_{12}+m_{13}+m_{23}}$. This proves the statement of the Proposition for $C_{m_{12},m_{13},m_{23}}^-\big(\widetilde\kappa,\widetilde\lambda\big)$. Consideration for $C_{m_{12},m_{13},m_{23}}^{+}\big(\widetilde \kappa,\widetilde \lambda\big)$ is completely analogous.
\end{proof}

\begin{definition}
We refer to the following pair of formal power series in $\widetilde X_{12},\widetilde X_{13},\widetilde X_{23}$ as a leading terms at $\widetilde\lambda=\widetilde\kappa$ of the Genus two Cauchy sum and its logarithmic derivative.
\begin{subequations}
\begin{align}
\Omega_-\big(\widetilde X_{12},\widetilde X_{13},\widetilde X_{23}\big|\widetilde\kappa\big):=&\mathop{\mathrm{Res}}_{\widetilde\lambda=\widetilde\kappa}\,(1-\widetilde\lambda/\widetilde\kappa)\,\mathcal C_-\big(\widetilde X_{12},\widetilde X_{13},\widetilde X_{23}\big|\widetilde\kappa,\widetilde\lambda\big)\qquad\in\quad\mathbb C\big(\widetilde\kappa\big)\big\llbracket\widetilde X_{12},\widetilde X_{13},\widetilde X_{23}\big\rrbracket,\label{eq:OmegaMinusDefinition}\\[0.3em]
\Omega_+\big(\widetilde X_{12},\widetilde X_{13}\widetilde X_{23}\big|\widetilde\kappa\big):=&\mathop{\mathrm{Res}}_{\widetilde\lambda=\widetilde\kappa}(1-\widetilde\lambda/\widetilde\kappa)^2\mathcal C_+\big(\widetilde X_{12},\widetilde X_{13},\widetilde X_{23}|\widetilde\kappa,\widetilde\lambda\big)\qquad\in\quad\mathbb C\big(\widetilde\kappa\big)\big\llbracket\widetilde X_{12},\widetilde X_{13},\widetilde X_{23}\big\rrbracket.\label{eq:OmegaPlusDefinition}
\end{align}
\end{subequations}
\end{definition}

\begin{lemma}
We have the following differential equation relating the leading terms at $\widetilde\kappa=\widetilde\lambda$ of the Genus two Cauchy sum and its logarithmic derivative in $\kappa$
\begin{equation}
\Omega_+\big(\widetilde X_{12},\widetilde X_{13},\widetilde X_{23}\big|\widetilde\kappa\big)=\left(-2-d_{\widetilde X}\right)\Omega_-\big(\widetilde X_{12},\widetilde X_{13},\widetilde X_{23}\big|\widetilde\kappa\big),
\label{eq:LogarithmicDerivativeLeadingTermsRelation}
\end{equation}
where $d_{\widetilde X}$ is a logarithmic derivative in $\widetilde X$-variables,
\begin{equation*}
d_{\widetilde X}:=\widetilde X_{12}\frac\partial{\partial\widetilde X_{12}}+\widetilde X_{13}\frac\partial{\partial\widetilde X_{13}}+\widetilde X_{23}\frac\partial{\partial\widetilde X_{23}}.
\end{equation*}
\end{lemma}
\begin{proof}
Change of variables (\ref{eq:ChangeOfVariablesXTilde}) corresponds to the following change of the first partial derivatives
\begin{align*}
\frac\partial{\partial\widetilde\kappa}=&\frac\partial{\partial\kappa}+\frac{\lambda}{\kappa^2}\big(1-\lambda/\kappa\big)^{-1}\left((x_{12}-1)\frac\partial{\partial x_{12}}+(x_{13}-1)\frac\partial{\partial x_{13}}+(x_{23}-1)\frac\partial{\partial x_{23}}\right),\\[0.3em]\qquad
\frac\partial{\partial\widetilde\lambda}=&\frac\partial{\partial\lambda}-\frac1\kappa\big(1-\lambda/\kappa\big)^{-1}\left((x_{12}-1)\frac\partial{\partial x_{12}}+(x_{13}-1)\frac\partial{\partial x_{13}}+(x_{23}-1)\frac\partial{\partial x_{23}}\right),\\[0.3em]
\frac\partial{\partial\widetilde X_{ij}}=&\big(1-\widetilde\lambda/\widetilde\kappa\big)\frac\partial{\partial x_{ij}},\qquad 1\leq i<j\leq3.
\end{align*}
Solving for partial derivative in $\kappa$ we get
\begin{equation}
\frac\partial{\partial\kappa}=\frac\partial{\partial\widetilde\kappa}-\frac{\widetilde\lambda}{\widetilde\kappa^2}\left(1-\widetilde\lambda/\widetilde\kappa\right)^{-1}d_{\widetilde X}.
\label{eq:ParitalKappaThroughTilde}
\end{equation}

On the other hand, by Proposition \ref{prop:PoleAtKappaEqualsLambda} we know that
\begin{subequations}
\begin{align}
\mathcal C_-\big(\widetilde X_{12},\widetilde X_{13},\widetilde X_{23}\big|\widetilde\kappa,\widetilde\lambda\big)=&\frac{\Omega_-\big(\widetilde X_{12},\widetilde X_{13},\widetilde X_{23}\big|\widetilde\kappa,\widetilde\lambda)}{\big(1-\widetilde\lambda/\widetilde\kappa\big)^2}+\mathbf O\left(\big(1-\widetilde\lambda/\widetilde\kappa\big)^{-1}\right),\label{eq:CMinusPoleDecomposition}\\[0.3em]
\mathcal C_+\big(\widetilde X_{12},\widetilde X_{13},\widetilde X_{23}\big|\widetilde\kappa,\widetilde\lambda\big)=&\frac{\Omega_+\big(\widetilde X_{12},\widetilde X_{13},\widetilde X_{23}\big|\widetilde\kappa,\widetilde\lambda)}{\big(1-\widetilde\lambda/\widetilde\kappa\big)^3
}+\mathbf O\left(\left(1-\widetilde\lambda/\widetilde\kappa\right)^{-2}\right),\label{eq:CPlusPoleDecomposition}
\end{align}
\label{eq:CPlusMinusLeadingTermSubstitution}
\end{subequations}
where by
\begin{equation*}
\mathbf O\left(\big(1-\widetilde\lambda/\widetilde\kappa\big)^{-k}\right),\qquad k\in\mathbb Z
\end{equation*}
we denote the formal power series in $\widetilde X_{12},\widetilde X_{13},\widetilde X_{23}$ with coefficients being rational functions in $\widetilde\kappa,\widetilde\lambda$ which when viewed as functions in $\widetilde\lambda$ have a pole at $\widetilde\lambda=\widetilde\kappa$ of order at most $k$. Note that we allow coefficients of the formal power series to have other poles.

Substituting (\ref{eq:ParitalKappaThroughTilde}) and (\ref{eq:CPlusMinusLeadingTermSubstitution}) into (\ref{eq:CPlusDefinition}) we get
\begin{align*}
\frac{\Omega_+\big(\widetilde X_{12},\widetilde X_{13},\widetilde X_{23}\big|\widetilde\kappa\big)}{\big(1-\widetilde\lambda/\widetilde\kappa\big)^3}=&\left(\widetilde\kappa\frac\partial{\partial\widetilde\kappa}-\frac{\widetilde\lambda}{\widetilde\kappa}\big(1-\widetilde\lambda/\widetilde\kappa\big)^{-1}d_{\widetilde X}\right)\frac{\Omega_-(\widetilde X_{12},\widetilde X_{13},\widetilde X_{23}|\widetilde\kappa)}{\big(1-\widetilde\lambda/\widetilde\kappa\big)^2}+\mathbf O\left(\big(1-\widetilde\lambda/\widetilde\kappa\big)^{-2}\right)\\[0.3em]
=&-2\frac{\widetilde\lambda}{\widetilde\kappa}\frac{\Omega_-\big(\widetilde X_{12},\widetilde X_{13},\widetilde X_{23}\big|\widetilde\kappa\big)}{\big(1-\widetilde\lambda/\widetilde\kappa\big)^3}-\frac{\widetilde\lambda}{\widetilde\kappa}\big(1-\widetilde\lambda/\widetilde\kappa\big)^{-3}d_{\widetilde X}\Omega_-\big(\widetilde X_{12},\widetilde x_{13},\widetilde X_{23}\big|\widetilde\kappa\big)+\mathbf O\left(\big(1-\widetilde\lambda/\widetilde\kappa\big)^{-2}\right).
\end{align*}
Multiplying both sides by $\big(1-\widetilde\lambda/\widetilde\kappa\big)^2$ and taking the residue around $\widetilde\lambda=\widetilde\kappa$ on both sides we get exactly (\ref{eq:LogarithmicDerivativeLeadingTermsRelation}).
\end{proof}

\begin{lemma}
The leading term at $\widetilde\lambda=\widetilde\kappa$ of the Genus two Cauchy sum and its logarithmic derivative satisfy the following differential equation
\begin{equation}
\hat{\mathcal H}_1^{(-2)}\Omega_-\big(\widetilde X_{12},\widetilde X_{13},\widetilde X_{23}\big|\widetilde\kappa\big)=\left(-3-d_{\widetilde X}\right)\Omega_+\big(\widetilde X_{12},\widetilde X_{13},\widetilde X_{23}\big|\widetilde\kappa\big),
\label{eq:H1m2TildeOmegaMinus}
\end{equation}
where $\hat{\mathcal H}_1^{(-2)}$ is a homogeneous differential operators of degree $-2$ defined as
\begin{equation}
\hat{\mathcal H}_1^{(-2)}:=\frac{\partial^2}{\partial\widetilde X_{12}^2}+\frac{\partial^2}{\partial\widetilde X_{13}^2}+\frac{\widetilde X_{12}^2+\widetilde X_{13}^2-\widetilde X_{23}^2}{\widetilde X_{12}\widetilde X_{13}}\frac{\partial^2}{\partial\widetilde X_{12}\partial\widetilde X_{13}}+\frac2{\widetilde X_{12}}\frac{\partial}{\widetilde\partial X_{12}}+\frac2{\widetilde X_{13}}\frac{\partial}{\widetilde\partial X_{13}}.
\label{eq:H1m2CalDefinition}
\end{equation}
\end{lemma}
\begin{proof}
First, comparing (\ref{eq:ChangeOfVariablesxToX}) with (\ref{eq:ChangeOfVariablesXTilde}) we get
\begin{equation}
X_{i,j}=\big(1-\widetilde\lambda/\widetilde\kappa\big)\widetilde X_{i,j},\qquad 
\frac\partial{\partial X_{i,j}}=\big(1-\widetilde\lambda/\widetilde\kappa\big)^{-1}\frac{\partial}{\partial\widetilde X_{i,j}},\qquad 1\leq i<j\leq3.
\label{eq:ChangeXtoTildeX}
\end{equation}

Next, substituting (\ref{eq:ChangeXtoTildeX}) into (\ref{eq:H1m2}) we obtain the following expression homogeneous components of the differential operator $\hat H_1$ in terms of new variables
\begin{equation}
\hat H_1^{(m)}=\big(1-\widetilde\lambda/\widetilde\kappa\big)^m\hat{\mathcal H}_1^{(m)},\qquad m\geq-2,
\label{eq:H1mThroughCal}
\end{equation}
where $\hat{\mathcal H}_1^{(m)}$ stands for the homogeneous differential operator in $\widetilde X_{12},\widetilde X_{13},\widetilde X_{23}$ of degree $m$ with constant coefficients. In particular, for $\hat{\mathcal H}_1^{(m)}$ this reads (\ref{eq:H1m2CalDefinition}).

Substituting (\ref{eq:CPlusMinusLeadingTermSubstitution}) and (\ref{eq:HkHomogeneousComponentDecompsition}) along with (\ref{eq:H1mThroughCal}) on into (\ref{eq:H1CMinus}) we get
\begin{align*}
\big(1-\widetilde\lambda/\widetilde\kappa\big)^{-4}\hat{\mathcal H}_1^{(-2)}\Omega_-\big(\widetilde X_{12},\widetilde X_{13},\widetilde X_{23}\big|\widetilde\kappa\big)&+\mathbf O\left(\big(1-\widetilde\lambda/\widetilde\kappa\big)^{-3}\right)\\[0.3em]
=\;\;&\kappa\frac\partial{\partial\kappa}\left(\frac{\Omega_+(\widetilde X_{12},\widetilde X_{13},\widetilde X_{23}\big|\widetilde\kappa)}{\big(1-\widetilde\lambda/\widetilde\kappa\big)^3}+\mathbf O\left(\big(1-\widetilde\lambda/\widetilde\kappa\big)^{-2}\right)\right)\\[0.3em]
\stackrel{(\ref{eq:ParitalKappaThroughTilde})}{=}\,&\frac{\widetilde\lambda}{\widetilde\kappa}\big(1-\widetilde\lambda/\widetilde\kappa\big)^{-4}\left(-3-d_{\widetilde X}\right)\Omega_+\big(\widetilde X_{12},\widetilde X_{13},\widetilde X_{23}\big|\widetilde\kappa\big)+\mathbf O\left(\big(1-\widetilde\lambda/\widetilde\kappa\big)^{-3}\right).
\end{align*}
Multiplying by $\big(1-\widetilde\lambda/\widetilde\kappa\big)^3$ and taking residue at $\widetilde\lambda=\widetilde\kappa$ we arrive at (\ref{eq:H1m2TildeOmegaMinus}).
\end{proof}

\begin{proposition}
The leading terms at $\widetilde\lambda=\widetilde\kappa$ of the Genus two Cauchy sum and its logarithmic derivative satisfy the following partial differential equations
\begin{subequations}
\begin{align}
\left(\hat{\mathcal H}_1^{(-2)}-d_{\widetilde X}^2-5d_{\widetilde X}-6\right)\Omega_-\big(\widetilde X_{12},\widetilde X_{13},\widetilde X_{23}\big|\widetilde\kappa\big)=&0,
\label{eq:DifferentialEquationOmegaMinus}\\[0.5em]
\left(\hat{\mathcal H}_1^{(-2)}-d_{\widetilde X}^2-7d_{\widetilde X}-12\right)\Omega_+\big(\widetilde X_{12},\widetilde X_{13},\widetilde X_{23}\big|\widetilde\kappa\big)=&0.
\label{eq:DifferentialEquationOmegaPlus}
\end{align}
\label{eq:DifferentialEquations}
\end{subequations}
\label{prop:DifferentialEquationsMain}
\end{proposition}
\begin{proof}
Substitute (\ref{eq:LogarithmicDerivativeLeadingTermsRelation}) into (\ref{eq:H1m2TildeOmegaMinus}).
\end{proof}

\begin{lemma}
The following formula holds:
\begin{equation}
\phi_{j_1,j_2,j_1-j_2}(x_{12},x_{13},1) = \sum\limits_{a = 0}^{j_2} \sum\limits_{b = 0}^{j_1-j_2} \ c_{a,b} \ (x_{12}+x_{12}^{-1})^{j_2-a} \ (x_{13}+x_{13}^{-1})^{j_1-j_2-b}
\label{AnsatzSpecialization}
\end{equation}
where
\begin{equation}
c_{a,b} = (-1)^{a/2+b/2} \dfrac{j_2!}{a!(j_2-a)!} \dfrac{(j_1-j_2)!}{b!(j_1-j_2-b)!} \dfrac{(j_1-a/2-b/2)!}{(j_1+1)!} \dfrac{(a+b)!}{(a/2+b/2)!}, \ \ \ \ \ a \geq 0, b \geq 0, a+b \mod 2 = 0
\label{AnsatzSpecializationCab}
\end{equation}
and $c_{a,b} = 0$ if either $a < 0$ or $b < 0$ or $a+b \mod 2 = 1$.
\begin{proof}
The fact that expansion of $\phi_{j_1,j_2,j_1-j_2}(x_{12},x_{13},1)$ has the form (\ref{AnsatzSpecialization}) follows from Lemma \ref{lemm:LeadingTermSchur}: indeed, for $j_3 = j_1 - j_2$ we have $d_1=0, d_2 = j_1 - j_2, d_3 = j_1$. It remains to prove that coefficients $c_{a,b}$ have the form (\ref{AnsatzSpecializationCab}). For this, we use the differential operator
\begin{equation*}
    \hat H_1\big|_{x_{23}=1} = x_{12}^2 \frac{\partial^2}{\partial x_{12}^2}+x_{13}^2\frac{\partial^2}{\partial x_{13}^2}+\frac{2(x_{12}^2+1)(x_{13}^2+1) - 2x_{12}x_{13}}{(x_{12}-x_{12}^{-1})(x_{13}-x_{13}^{-1})} \frac{\partial^2}{\partial x_{12}\partial x_{13}}+\frac{3x_{12}^2+1}{x_{12}-x_{12}^{-1}}\frac{\partial}{\partial x_{12}}+\frac{3x_{13}^2+1}{x_{13}-x_{13}^{-1}}\frac{\partial}{\partial x_{13}}+1
\end{equation*}
Proposition \ref{prop:GenusTwoSchurHiEigenfunctions} implies that $\phi_{j_1,j_2,j_1-j_2}(x_{12},x_{13},1)$ is an eigenfunction of $\hat H_1\big|_{x_{23}=1}$ with eigenvalue $(j_1+1)^2$. In addition, Lemma \ref{lemm:LeadingTermSchur} implies that it is the unique (up to normalization) eigenfunction with this eigenvalue and leading term $x_{12}^{j_2} x_{13}^{j_1-j_2}$. To prove the formula for coefficients $c_{a,b}$ it remains to show that eq.(\ref{AnsatzSpecialization}) with $c_{a,b}$ given by eq. (\ref{AnsatzSpecializationCab}) is an eigenfunction with correct eigenvalue, and has correct normalization.

To show this, we start with
\begin{align*}
 \hat H_1\big|_{x_{23}=1} (x_{12}+x_{12}^{-1})^{u} \ (x_{13}+x_{13}^{-1})^{v} =& (u+v+1)^2 (x_{12}+x_{12}^{-1})^{u} \ (x_{13}+x_{13}^{-1})^{v} \\
& -4 u (u-1) (x_{12}+x_{12}^{-1})^{u-2} \ (x_{13}+x_{13}^{-1})^{v} \\
& -8 uv (x_{12}+x_{12}^{-1})^{u-1} \ (x_{13}+x_{13}^{-1})^{v-1} \\ 
& - 4 v (v-1) (x_{12}+x_{12}^{-1})^{u} \ (x_{13}+x_{13}^{-1})^{v-2}
\end{align*}
which holds by direct computation. Therefore,

\begin{align*}
\hat H_1\big|_{x_{23}=1} \phi_{j_1,j_2,j_1-j_2}(x_{12},x_{13},1) = \sum\limits_{a = 0}^{j_2} \sum\limits_{b = 0}^{j_1-j_2} \ {\widetilde c}_{a,b} \ (x_{12}+x_{12}^{-1})^{j_2-a} \ (x_{13}+x_{13}^{-1})^{j_1-j_2-b}
\end{align*}
with

\begin{align*}
{\widetilde c}_{a,b} =& (j_1-a-b+1)^2 c_{a,b} - 4(j_2-a+2)(j_2-a+1) c_{a-2,b} \\
& - 8 (j_2-a+1)(j_1-j_2-b+1) c_{a-1,b-1} - 4(j_1-j_2-b+2)(j_1-j_2-b+1) c_{a,b-2}
\end{align*}
We need to show that

\begin{equation*}
{\widetilde c}_{a,b} = (j_1+1)^2 c_{a,b}
\end{equation*}
and this amounts to the following recursive relation for coefficients

\begin{align*}
& (a+b)(2 j_1 -a-b+2) c_{a,b} + 4(j_2-a+2)(j_2-a+1) c_{a-2,b} \\
& + 8 (j_2-a+1)(j_1-j_2-b+1) c_{a-1,b-1} + 4(j_1-j_2-b+2)(j_1-j_2-b+1) c_{a,b-2} = 0
\end{align*}
which holds by direct computation. To verify normalization, we compute

\begin{equation*}
\phi_{j_1,j_2,j_1-j_2}(1,1,1) = \sum\limits_{a = 0}^{j_2} \sum\limits_{b = 0}^{j_1-j_2} \ c_{a,b} \ 2^{j_1-a-b} 
\end{equation*}
Changing summation variables via $a = 2 s - b$, we find

\begin{equation*}
\phi_{j_1,j_2,j_1-j_2}(1,1,1) = \sum\limits_{s = 0}^{\lfloor j_1/2 \rfloor} \sum\limits_{b = 0}^{\infty} \ c_{2s-b,b} \ 2^{j_1-2s} 
\end{equation*}
where the upper limit of summation for the $b$ variable can be put to infinity because coefficients $c_{a,b}$ vanish if either $a > j_2$ or $b > j_1 - j_2$ or $b > 2s$. The sum over $b$ can be computed using the Gauss summation theorem, and we find

\begin{equation*}
\phi_{j_1,j_2,j_1-j_2}(1,1,1) = \sum\limits_{s = 0}^{j_1/2} \dfrac{2^{j_1 - 2 s} (-1)^s \Gamma(j_1 + 1 - s)}{(j_1 + 1) \Gamma(j_1 - 2 s + 1) \Gamma(s + 1)}
\end{equation*}
The sum over $s$ is also computed using the Gauss summation theorem, and we find

\begin{equation*}
\sum\limits_{s = 0}^{j_1/2} \dfrac{2^{j_1 - 2 s} (-1)^s \Gamma(j_1 + 1 - s)}{(j_1 + 1) \Gamma(j_1 - 2 s + 1) \Gamma(s + 1)} = 1
\end{equation*}
Therefore, the normalization is correct.
\end{proof}
\label{lemm:SpecializationCab}
\end{lemma}

\begin{proposition}
Functions $\phi_{j_1,j_2,j_3}$ satisfy
\begin{equation*}
\sum\limits_{j_2 + j_3 = J} \phi_{j_1,j_2,j_3}(x_{12},x_{13},1) = \dfrac{(1-x_{13}^{j_1+1}/x_{12}^{j_1+1})(1-x_{13}^{j_1+1} x_{12}^{j_1+1})}{(j_1+1)x_{13}^{j_1}(1-x_{13}/x_{12})(1-x_{13}x_{12})}
\end{equation*}
\label{prop:SpecializedSum}
\end{proposition}
\begin{proof}
We start from the 3-rd Pieri identity in the form
\begin{align*}
& (j_2+1)(j_3+1) K_{1,-1}(j_2,j_3,j_1) \phi_{j_1,j_2+1,j_3-1}(x_{12},x_{13},1) 
\\ & +
(j_2+1)(j_3+1) K_{-1,1}(j_2,j_3,j_1) \phi_{j_1,j_2-1,j_3+1}(x_{12},x_{13},1) \\ & - 2 (j_2+1)(j_3+1) \phi_{j_1,j_2,j_3}(x_{12},x_{13},1) \\
  & = - (j_2+1)(j_3+1) K_{1,1}(j_2,j_3,j_1) \phi_{j_1,j_2+1,j_3+1}(x_{12},x_{13},1) 
\\ & - (j_2+1)(j_3+1) K_{-1,-1}(j_2,j_3,j_1) \phi_{j_1,j_2-1,j_3-1}(x_{12},x_{13},1)
\end{align*}
The sum of the l.h.s. equals to
\begin{align*}
& \sum\limits_{j_2+j_3=J} \left( \ (j_2+1)(j_3+1) K_{1,-1}(j_2,j_3,j_1) \phi_{j_1,j_2+1,j_3-1}(x_{12},x_{13},1) \right.
\\ & +
(j_2+1)(j_3+1) K_{-1,1}(j_2,j_3,j_1) \phi_{j_1,j_2-1,j_3+1}(x_{12},x_{13},1) \\ & \left. - 2 (j_2+1)(j_3+1) \phi_{j_1,j_2,j_3}(x_{12},x_{13},1) \right) \\
& = \left(\frac{j_1^2}{2} + j_1 - \frac{(J+2)^2}{2}\right) \sum\limits_{j_2+j_3=J} \phi_{j_1,j_2,j_3}(x_{12},x_{13},1) -
\frac{1}{4} (j_1 + J + 2)(j_1 - J) \phi_{j_1, 0, J}(x_{12},x_{13},1)\\
&-\frac{1}{4} (j_1 + J + 2)(j_1 - J) \phi_{j_1, J, 0}(x_{12},x_{13},1)
\end{align*}
At the same time the sum of the r.h.s. equals to
\begin{align*}
\sum\limits_{j_2+j_3=J} \left( \ - (j_2+1)(j_3+1) K_{1,1}(j_2,j_3,j_1) \phi_{j_1,j_2+1,j_3+1}(x_{12},x_{13},1)\right.\\
\left.- (j_2+1)(j_3+1) K_{-1,-1}(j_2,j_3,j_1) \phi_{j_1,j_2-1,j_3-1}(x_{12},x_{13},1) \ \right) \\
= \left(\frac{j_1^2}{4} + \frac{j_1}{2} - \frac{(J+2)(J+4)}{4}\right) \sum\limits_{j_2+j_3=J+2} \phi_{j_1,j_2,j_3}(x_{12},x_{13},1)\\
+ \left(\frac{j_1^2}{4} + \frac{j_1}{2} - \frac{J(J+2)}{4}\right) \sum\limits_{j_2+j_3=J-2} \phi_{j_1,j_2,j_3}(x_{12},x_{13},1) \\
-
\frac{1}{4} (j_1 + J + 4)(j_1 - J - 2) \phi_{j_1, 0, J + 2}(x_{12},x_{13},1)\\
-\frac{1}{4} (j_1 + J + 4)(j_1 - J - 2) \phi_{j_1, J + 2, 0}(x_{12},x_{13},1)
\end{align*}
Note that the terms proportional to $\phi$-functions with one zero index all vanish, because of vanishing prefactors. We therefore find a recursive relation
\begin{align*}
-\left(\frac{j_1^2}{4} + \frac{j_1}{2} - \frac{(J+2)(J+4)}{4}\right) \sum\limits_{j_2+j_3=J+2} \phi_{j_1,j_2,j_3}(x_{12},x_{13},1) \\ +
\left(\frac{j_1^2}{2} + j_1 - \frac{(J+2)^2}{2}\right) \sum\limits_{j_2+j_3=J} \phi_{j_1,j_2,j_3}(x_{12},x_{13},1) 
\\ - \left(\frac{j_1^2}{4} + \frac{j_1}{2} - \frac{J(J+2)}{4}\right) \sum\limits_{j_2+j_3=J-2} \phi_{j_1,j_2,j_3}(x_{12},x_{13},1) = 0
\end{align*}
Specializing here $J=j_1$ we find that

\begin{equation}
\sum\limits_{j_2+j_3=j_1+2} \phi_{j_1,j_2,j_3}(x_{12},x_{13},1) = \sum\limits_{j_2+j_3=j_1} \phi_{j_1,j_2,j_3}(x_{12},x_{13},1)
\end{equation}
With these initial conditions, this second order recursion has a unique solution, namely for all $J$ we have

\begin{equation*}
\sum\limits_{j_2+j_3=J} \phi_{j_1,j_2,j_3}(x_{12},x_{13},1) = \sum\limits_{j_2+j_3=j_1} \phi_{j_1,j_2,j_3}(x_{12},x_{13},1)
\end{equation*}
Therefore, to prove the current proposition, it is sufficient to prove it for $J=j_1$. For this, note that

\begin{equation*}
\dfrac{(1-x_{13}^{j_1+1}/x_{12}^{j_1+1})(1-x_{13}^{j_1+1} x_{12}^{j_1+1})}{(j_1+1)x_{13}^{j_1}(1-x_{13}/x_{12})(1-x_{13}x_{12})}
\end{equation*}
is an eigenfunction of $\hat H_1\big|_{x_{23}=1}$ with eigenvalue $(j_1+1)^2$. At the same time, functions $\phi_{j_1,j_2,j_1-j_2}(x_{12},x_{13},1)$ for $j_2 = 0, \ldots, j_1$ form a basis in the space of eigenfunctions of $\hat H_1\big|_{x_{23}=1}$ with eigenvalue $(j_1+1)^2$. Therefore, there exist coefficients $c_{j_2}$ such that

\begin{equation*}
\sum\limits_{j_2=0}^{j_1} c_{j_2} \phi_{j_1,j_2,j_1-j_2}(x_{12},x_{13},1) = \dfrac{(1-x_{13}^{j_1+1}/x_{12}^{j_1+1})(1-x_{13}^{j_1+1} x_{12}^{j_1+1})}{(j_1+1)x_{13}^{j_1}(1-x_{13}/x_{12})(1-x_{13}x_{12})}
\end{equation*}
To find $c_{j_2}$, note that expansion of the r.h.s. in monomials in $x_{12}, x_{13}$ contains the monomial $x_{12}^{j_2} x_{13}^{j_1-j_2}$ which is a leading term in $\phi_{j_1,j_2,j_1-j_2}(x_{12},x_{13},1)$, with coefficient $1/(j_1+1)$. At the same time, the coefficient in front of the leading term in $\phi_{j_1,j_2,j_1-j_2}(x_{12},x_{13},1)$ is given by $c_{0,0} = 1/(j_1+1)$. Hence, all $c_{j_2} = 1$, what completes the proof.

\end{proof}

\begin{corollary}
The leading terms at $\widetilde\kappa=\widetilde\lambda$ of the Genus two Cauchy sum and its logarithmic derivative when specialized to $X_{23}=0$ are given by the following formula
\begin{subequations}
\begin{align}
\Omega_-\big(\widetilde X_{12},\widetilde X_{13},0\,\big|\,\widetilde\kappa\big)=&\dfrac{1}{\widetilde\kappa(\widetilde\kappa^2-1)} \dfrac{-2 \ \mathop{\mathrm{arctanh}}\left(\dfrac{\widetilde X_{12}^2-\widetilde X_{13}^2}{\widetilde X_{12}^2+\widetilde X_{13}^2-2}\right)}{\widetilde X_{12}^2-\widetilde X_{13}^2},
\label{eq:InitialConditionsOmegaMinus}\\[0.5em]
\Omega_+\big(\widetilde X_{12},\widetilde X_{13},0\,\big|\,\widetilde\kappa\big)=&\dfrac{1}{\widetilde\kappa(\widetilde\kappa^2-1)} \dfrac{-2}{(\widetilde X_{12}^2-1)(\widetilde X_{13}^2-1)}.
\label{eq:InitialConditionsOmegaPlus}
\end{align}
\label{eq:InitialConditions}
\end{subequations}
\end{corollary}
\begin{proof}
Proposition (\ref{prop:SpecializedSum}) implies that
\begin{align*}
& \mathcal C_+(x_{12},x_{13},1|\kappa,\lambda)
=\sum_{(j_1,j_2,j_3)\in\mathbf J}\phi_{j_1,j_2,j_3}(x_{12},x_{13},1) (j_1+1)(\kappa^{j_1+1}+\kappa^{-j_1-1})\lambda^{j_2+j_3} \\ & = \sum\limits_{j_1 = 0}^{\infty} \dfrac{(1-x_{13}^{j_1+1}/x_{12}^{j_1+1})(1-x_{13}^{j_1+1} x_{12}^{j_1+1})}{x_{13}^{j_1}(1-x_{13}/x_{12})(1-x_{13}x_{12})} (\kappa^{j_1+1}+\kappa^{-j_1-1}) \dfrac{\lambda^{j_1}}{1-\lambda^2} 
\end{align*}
Taking the sum of the geometric series, we find

\begin{align*}
& \mathcal C_+(x_{12},x_{13},1|\kappa,\lambda)= \dfrac{1}{(\lambda \kappa x_{13} - 1) (\lambda \kappa x_{12} - 1) (\lambda \kappa - x_{13}) (\lambda x_{13} - \kappa) (-\kappa x_{13} + \lambda) (\lambda \kappa - x_{12}) (\lambda x_{12} - \kappa) (-\kappa x_{12} + \lambda)} \\ & (\lambda^4 \kappa^4 x_{12} x_{13} - \lambda^2 \kappa^6 x_{12} x_{13} + \lambda^2 \kappa^4 x_{12}^2 x_{13}^2 - 2 \lambda^3 \kappa^3 x_{12}^2 x_{13} - 2 \lambda^3 \kappa^3 x_{12} x_{13}^2 + \lambda^4 \kappa^2 x_{12} x_{13} + \lambda^2 \kappa^4 x_{12}^2 + 3 \lambda^2 \kappa^4 x_{12} x_{13} \\ & + \lambda^2 \kappa^4 x_{13}^2 + \lambda^2 \kappa^2 x_{12}^2 x_{13}^2 - 2 \lambda^3 \kappa^3 x_{12} - 2 \lambda^3 \kappa^3 x_{13} - 2 \lambda \kappa^3 x_{12}^2 x_{13} - 2 \lambda \kappa^3 x_{12} x_{13}^2 + \lambda^2 \kappa^4 + \lambda^2 \kappa^2 x_{12}^2 + 3 \lambda^2 \kappa^2 x_{12} x_{13} \\ & + \lambda^2 \kappa^2 x_{13}^2 + \kappa^4 x_{12} x_{13} - 2 \lambda \kappa^3 x_{12} - 2 \lambda \kappa^3 x_{13} + \lambda^2 \kappa^2 - \lambda^2 x_{12} x_{13} + \kappa^2 x_{12} x_{13}) \kappa x_{13} x_{12}
\end{align*}
Computing the residue, we find

\begin{align*}
\Omega_+\big(\widetilde X_{12},\widetilde X_{13}, 0\big|\widetilde\kappa\big):=&\mathop{\mathrm{Res}}_{\widetilde\lambda=\widetilde\kappa}(1-\widetilde\lambda/\widetilde\kappa)^2\mathcal C_+\big(\widetilde X_{12},\widetilde X_{13},0|\widetilde\kappa,\widetilde\lambda\big)=\dfrac{1}{\widetilde\kappa(\widetilde\kappa^2-1)} \dfrac{-2}{(\widetilde X_{12}^2-1)(\widetilde X_{13}^2-1)}
\end{align*}
Similarly, for the other case we find

\begin{align*}
& \mathcal C_-(x_{12},x_{13},1|\kappa,\lambda)
=\sum_{(j_1,j_2,j_3)\in\mathbf J}\phi_{j_1,j_2,j_3}(x_{12},x_{13},1) (\kappa^{j_1+1}-\kappa^{-j_1-1})\lambda^{j_2+j_3} \\ & = \sum\limits_{j_1 = 0}^{\infty} \dfrac{(1-x_{13}^{j_1+1}/x_{12}^{j_1+1})(1-x_{13}^{j_1+1} x_{12}^{j_1+1})}{(j_1+1) x_{13}^{j_1}(1-x_{13}/x_{12})(1-x_{13}x_{12})} (\kappa^{j_1+1}-\kappa^{-j_1-1}) \dfrac{\lambda^{j_1}}{1-\lambda^2} 
\end{align*}
Taking the sum of the logarithmic series, we find

\begin{equation*}
\mathcal C_-(x_{12},x_{13},1|\kappa,\lambda)= \dfrac{x_{12}x_{13}}{(x_{12}-x_{13})(1-x_{12}x_{13}) \lambda (1-\lambda^2)} \ln \left( \dfrac{(1 - \lambda \kappa x_{12}) (x_{12} - \lambda\kappa) (\kappa x_{13} - \lambda) (\kappa - \lambda x_{13})}{(\kappa - \lambda x_{12}) (\kappa x_{12} - \lambda) (x_{13}- \lambda \kappa)(1 - \lambda \kappa x_{13})} \right)
\end{equation*}
Computing the residue and rewriting logarithms using the arctanh function, we find

\begin{align*}
\Omega_-\big(\widetilde X_{12},\widetilde X_{13}, 0\big|\widetilde\kappa\big):=&\mathop{\mathrm{Res}}_{\widetilde\lambda=\widetilde\kappa}(1-\widetilde\lambda/\widetilde\kappa)\mathcal C_-\big(\widetilde X_{12},\widetilde X_{13},0|\widetilde\kappa,\widetilde\lambda\big)=\dfrac{1}{\widetilde\kappa(\widetilde\kappa^2-1)} \dfrac{-2 \ \mathop{\mathrm{arctanh}}\left(\dfrac{\widetilde X_{12}^2-\widetilde X_{13}^2}{\widetilde X_{12}^2+\widetilde X_{13}^2-2}\right)}{\widetilde X_{12}^2-\widetilde X_{13}^2}
\end{align*}

\end{proof}

\begin{theorem}
The leading term at $\widetilde\kappa=\widetilde\lambda$ of the Genus two Cauchy sum is given by the following formula
\begin{equation}
\begin{aligned}
\Omega_-\big(\widetilde X_{12},\widetilde X_{13},\widetilde X_{23}\,\big|\,\widetilde\kappa\big)=\dfrac{1}{\widetilde\kappa(\widetilde\kappa^2-1)}  \dfrac{-2\mathop{\mathrm{arctanh}}\left(\dfrac{\sqrt{\widetilde X_{12}^4 - 2 \widetilde X_{12}^2 \widetilde X_{13}^2 - 2 \widetilde X_{12}^2 \widetilde X_{23}^2 + \widetilde X_{13}^4 - 2 \widetilde X_{13}^2 \widetilde X_{23}^2 + \widetilde X_{23}^4 + 4 \widetilde X_{23}^2}}{\widetilde X_{12}^2 + \widetilde X_{13}^2 - \widetilde X_{23}^2 - 2}\right)}{\sqrt{\widetilde X_{12}^4 - 2 \widetilde X_{12}^2 \widetilde X_{13}^2 - 2 \widetilde X_{12}^2 \widetilde X_{23}^2 + \widetilde X_{13}^4 - 2 \widetilde X_{13}^2 \widetilde X_{23}^2 + \widetilde X_{23}^4 + 4 \widetilde X_{23}^2}}\\[0.5em]
\end{aligned}
\label{eq:LeadingTermOmegaMinus}
\end{equation}
\label{th:Arctangent}
\end{theorem}
\begin{proof}
By Definition \ref{def:GenusTwoSchurPieri}, genus two Schur polynomials $\phi_{j_1,j_2,j_3}$ are symmetric in $x_{23}\leftrightarrow x_{23}^{-1}$. Hence, by (\ref{eq:CMinusDefinition}) we get
\begin{equation*}
\frac\partial{\partial x_{23}}\mathcal C_-(x_{12},x_{13},x_{23}|\kappa,\lambda)\Big|_{x_{23}=1}=0.
\end{equation*}
Now using (\ref{eq:OmegaMinusDefinition}) combined with (\ref{eq:ChangeOfVariablesxToX}) and (\ref{eq:ChangeXtoTildeX}) we conclude
\begin{equation}
\frac\partial{\partial\widetilde X_{23}}\Omega_-\big(\widetilde X_{12},\widetilde X_{13},\widetilde X_{23}\,\big|\,\widetilde\kappa\big)\Big|_{\widetilde X_{23}=0}=0.
\label{eq:OmegaMinusZeroNormalDerivative}
\end{equation}

Next, by substitution we verify that the right hand side of (\ref{eq:LeadingTermOmegaMinus}) is a solution of the Cauchy problem for the second order partial differential equation (\ref{eq:DifferentialEquationOmegaMinus}) with initial conditions on the plane $\widetilde X_{23}=0$ specified by (\ref{eq:InitialConditionsOmegaMinus}) and (\ref{eq:OmegaMinusZeroNormalDerivative}). Because solution of the Cauchy problem is unique, the right hand side of (\ref{eq:LeadingTermOmegaMinus}) must coincide with $\Omega_-\big(\widetilde X_{12},\widetilde X_{13},\widetilde X_{23}\,\big|\,\widetilde\kappa\big)$.
\end{proof}

\begin{corollary}
The leading term of the logarithmic derivative of the Genus two Cauchy sum is given by
\begin{equation*}
\begin{aligned}
\Omega_+\big(\widetilde X_{12},\widetilde X_{13},\widetilde X_{23}\,\big|\,\widetilde\kappa\big)=\dfrac{1}{\widetilde\kappa(\widetilde\kappa^2-1)} \dfrac{8 \widetilde X_{23}^2 \mathop{\mathrm{arctanh}}\left(\dfrac{\sqrt{\widetilde X_{12}^4 - 2 \widetilde X_{12}^2 \widetilde X_{13}^2 - 2 \widetilde X_{12}^2 \widetilde X_{23}^2 + \widetilde X_{13}^4 - 2 \widetilde X_{13}^2 \widetilde X_{23}^2 + \widetilde X_{23}^4 + 4 \widetilde X_{23}^2}}{\widetilde X_{12}^2 + \widetilde X_{13}^2 - \widetilde X_{23}^2 - 2}\right)}{(\widetilde X_{12}^4 - 2 \widetilde X_{12}^2 \widetilde X_{13}^2 - 2 \widetilde X_{12}^2 \widetilde X_{23}^2 + \widetilde X_{13}^4 - 2 \widetilde X_{13}^2 \widetilde X_{23}^2 + \widetilde X_{23}^4 + 4 \widetilde X_{23}^2)^{3/2}}\\
- \dfrac{1}{\widetilde\kappa(\widetilde\kappa^2-1)} \dfrac{2(\widetilde X_{12}^4 - 2 \widetilde X_{12}^2 \widetilde X_{13}^2 - \widetilde X_{12}^2 \widetilde X_{23}^2 + \widetilde X_{13}^4 - \widetilde X_{13}^2 \widetilde X_{23}^2 + 2 \widetilde X_{23}^2)}{(\widetilde X_{12}^2-1)(\widetilde X_{13}^2-1)(\widetilde X_{12}^4 - 2 \widetilde X_{12}^2 \widetilde X_{13}^2 - 2 \widetilde X_{12}^2 \widetilde X_{23}^2 + \widetilde X_{13}^4 - 2 \widetilde X_{13}^2 \widetilde X_{23}^2 + \widetilde X_{23}^4 + 4 \widetilde X_{23}^2)}.
\end{aligned}
\end{equation*}
\end{corollary}
\begin{proof}
Use (\ref{eq:LogarithmicDerivativeLeadingTermsRelation}).
\end{proof}
Also note this this is a unique solution to (\ref{eq:DifferentialEquationOmegaPlus}) with initial condition (\ref{eq:InitialConditionsOmegaPlus}).

\section{Multiple Schur sums.}

In this section, we conjecture a generalization of the genus two Cauchy sum where the summand is a product of multiple genus two Schur polynomials.

\begin{definition}
Multiple genus two Cauchy sum is defined by following formal power series in $\lambda$:
\begin{align*}
\nonumber & \mathcal{C}_-(x^{(1)}_{12}, x^{(1)}_{13}, x^{(1)}_{23}, \ldots, x^{(m)}_{12}, x^{(m)}_{13}, x^{(m)}_{23}| \kappa, \lambda) := \sum_{(j_1, j_2, j_3) \in \mathbf{J}} \prod\limits_{i = 1}^{m} \phi_{j_1, j_2, j_3}(x^{(i)}_{12}, x^{(i)}_{13}, x^{(i)}_{23}) \cdot (\kappa^{j_1 + 1} - \kappa^{-j_1 - 1}) \lambda^{j_2 + j_3} \\ & \in \quad \mathbb{C}[\kappa^{\pm 1}][\big(x_{12}^{(1)}\big)^{\pm 1}, \big(x_{13}^{(1)}\big)^{\pm 1}, \big(x_{23}^{(1)}\big)^{\pm 1}, \ldots, \big(x_{12}^{(m)}\big)^{\pm 1}, \big(x_{13}^{(m)}\big)^{\pm 1}, \big(x_{23}^{(m)}\big)^{\pm 1}] \llbracket \lambda \rrbracket. 
\end{align*}
\end{definition}

Similarly to the single Schur sum, we define the leading terms

\begin{definition}
We refer to the following formal power series in $\widetilde X_{12}^{(1)},\widetilde X_{13}^{(1)},\widetilde X_{23}^{(1)}, \ldots, \widetilde X_{12}^{(m)},\widetilde X_{13}^{(m)},\widetilde X_{23}^{(m)}$ as a leading term at $\widetilde \lambda = \widetilde \kappa$ of the multiple genus two Cauchy sum:
\begin{align*}
\nonumber & \Omega_-\big(\widetilde X_{12}^{(1)},\widetilde X_{13}^{(1)},\widetilde X_{23}^{(1)}, \ldots, \widetilde X_{12}^{(m)},\widetilde X_{13}^{(m)},\widetilde X_{23}^{(m)} \big|\widetilde\kappa\big):=&\mathop{\mathrm{Res}}_{\widetilde\lambda=\widetilde\kappa}\,(1-\widetilde\lambda/\widetilde\kappa)\,\mathcal{C}_-(x^{(1)}_{12}, x^{(1)}_{13}, x^{(1)}_{23}, \ldots, x^{(m)}_{12}, x^{(m)}_{13}, x^{(m)}_{23}| \kappa, \lambda) \\ 
& \in\quad\mathbb C\big(\widetilde\kappa\big)\big\llbracket\widetilde X_{12},\widetilde X_{13},\widetilde X_{23}\big\rrbracket
\end{align*}
\end{definition}

\begin{conjecture}
Leading term at $\widetilde \lambda = \widetilde \kappa$ of the multiple genus two Cauchy sum is given by the following formula
\begin{align*}
\nonumber & \Omega_-\big(\widetilde X_{12}^{(1)},\widetilde X_{13}^{(1)},\widetilde X_{23}^{(1)}, \ldots, \widetilde X_{12}^{(m)},\widetilde X_{13}^{(m)},\widetilde X_{23}^{(m)} \big|\widetilde\kappa\big) = \sum\limits_{{\vec i}^{(1)}, \ldots, {\vec i}^{(m)} = 0}^{\infty} \ \prod\limits_{a = 1}^{m} \widetilde X_{12}^{i^{(a)}_{12}} \widetilde X_{13}^{i^{(a)}_{13}} \widetilde X_{23}^{i^{(a)}_{23}} \\ \nonumber & (-1)^{i^{(1)}_{23}+\ldots+i^{(m)}_{23}} \ 4^{-(i^{(1)}_{12}+i^{(1)}_{13}+i^{(1)}_{23}+\ldots+i^{(m)}_{12}+i^{(m)}_{13}+i^{(m)}_{23})} \\ 
\nonumber &  \dfrac{\Gamma\left(2i^{(1)}_{12}+2i^{(1)}_{13}+2i^{(1)}_{23}+\ldots+2i^{(m)}_{12}+2i^{(m)}_{13}+2i^{(m)}_{23}+2\right)
\Gamma\left(\frac{3}{2}\right)^{m}
}{\Gamma\left(i^{(1)}_{12}+i^{(1)}_{13}+2i^{(1)}_{23}+\ldots+i^{(m)}_{12}+i^{(m)}_{13}+2i^{(m)}_{23}+2\right)\Gamma\left(i^{(1)}_{12}+i^{(1)}_{13}+i^{(1)}_{23}+\frac{3}{2}\right) \ldots
\Gamma\left(i^{(m)}_{12}+i^{(m)}_{13}+i^{(m)}_{23}+\frac{3}{2}\right)} \\
& \dfrac{\Gamma\left(i^{(1)}_{12}+i^{(1)}_{23}+\ldots+i^{(m)}_{12}+i^{(m)}_{23}+1\right)\Gamma\left(i^{(1)}_{13}+i^{(1)}_{23}+\ldots+i^{(m)}_{13}+i^{(m)}_{23}+1\right)}
{\Gamma\left(i^{(1)}_{12}+1\right)\Gamma\left(i^{(1)}_{13}+1\right)\Gamma\left(i^{(1)}_{23}+1\right)\ldots\Gamma\left(i^{(m)}_{12}+1\right)\Gamma\left(i^{(m)}_{13}+1\right)\Gamma\left(i^{(m)}_{23}+1\right)}
\end{align*}
Note that this is a $A$-hypergeometric function in the sence of Gelfand, Kapranov, Zelevinsky \cite{GelfandKapranovZelevinsky'1994}.
\label{conj:MainConjecture}
\end{conjecture}

\appendix

\section{Proof of Proposition \ref{prop:KernelHm2}.}
\label{sec:ProofPropositionKernelHm2}

\begin{definition}
Let $m,k,l\in\mathbb Z_{\geq0}$ be nonnegative integers satisfying $k+l\leq m$, denote
\begin{equation}
\mathfrak P_{m,k,l}:=X_{23}^m\mathcal P_k\left(\frac{X_{12}-X_{13}}{X_{23}}\right)\mathcal P_l\left(\frac{X_{12}+X_{13}}{X_{23}}\right)\qquad\in\quad \mathbb C[X_{12},X_{13},X_{23}],
\label{eq:PmklDefinition}
\end{equation}
where $\mathcal P_k$ and $\mathcal P_l$ stand for Legendre polynomials of degrees $k$ and $l$ respectively.
\end{definition}

\begin{lemma}
The following provides a basis on $\left(\mathbb C[X_{12},X_{13},X_{23}]\right)_m$, the space of homogeneous polynomials of degree $m$ in variables $X_{12},X_{13},X_{23}$:
\begin{equation*}
\mathfrak B=\left\{\mathfrak P_{m,k,l}\;|\;m,k,l\in\mathbb Z_{\geq0},\; k+l\leq m\right\}.
\end{equation*}
\end{lemma}
\begin{proof}
First, note that $\left(\mathbb C[X_{12},X_{13},X_{23}]\right)_m$ has a basis given by $X_{23}^{m-k-l}(X_{12}-X_{13})^k(X_{12}+X_{13})^l$ with $k+l\leq m$. Next, observe that $\mathfrak B$ is obtained by triangular change of basis from the latter.
\end{proof}

\begin{lemma}
Consider restriction of the operator $\hat H_1^{(-2)}$ on a subspace of homogeneous polynomials of degree $m\in\mathbb Z_{\geq 0}$, it has the following kernel
\begin{equation}
\ker\hat H_1^{(-2)}\big|_{\left(\mathbb C[X_{12},X_{13},X_{23}]\right)_m}=\mathrm{Span}\left\{ \mathfrak P_{m,l,l}\;\left|\;0\leq l\leq\left\lfloor\frac m2\right\rfloor\right.\right\}.
\label{eq:KerH1m2}
\end{equation}
\end{lemma}
\begin{proof}
To this end we make the following change of variables
\begin{equation}
(X_{12},X_{23},X_{13})\leftrightarrow (U,V,X_{23}),\qquad\textrm{where}\qquad U=\frac{X_{12}-X_{13}}{X_{23}},\quad V=\frac{X_{12}+X_{13}}{X_{23}}.
\label{eq:UVXChangeOfVariables}
\end{equation}
In terms of the new variables operator (\ref{eq:H1m2}) multiplied by $X_{12}X_{13}$ on the left acquires especially simple form
\begin{equation}
X_{12}X_{13}\hat H_1^{(-2)}=\left(\left(1-U^2\right)\frac{\partial^2}{\partial U^2}-2U\frac{\partial}{\partial U}\right)-\left(\left(1-V^2\right)\frac{\partial^2}{\partial V^2}-2V\frac{\partial}{\partial V}\right)=\hat{\mathcal D}_U-\hat{\mathcal D}_V,
\label{eq:DifferenceOfLegendreOperators}
\end{equation}
where $\hat{\mathcal D}_U$ and $\hat{\mathcal D}_V$ are nothing but Legendre differential operators in variables $U,V$. Note that the action of the right hand side of (\ref{eq:DifferenceOfLegendreOperators}) on $\mathbb C[U,V]$ is diagonal in the basis of products of Legendre polynomials
\begin{equation*}
\left\{\mathcal P_l(U)\mathcal P_k(V)\;\big|\; k,l\in\mathbb Z_{\geq0}\right\}
\end{equation*}
with eigenvalues $\lambda_{k,l}=k(k+1)-l(l+1)$.

As a corollary, the action of $X_{12}X_{13}\hat H_1^{(-2)}$ on the homogeneous component $\left(\mathbb C[X_{12},X_{13},X_{23}]\right)_m$ of degree $m$ is diagonal in basis $\mathfrak B$
with the same eigenvalues $\lambda_{k,l}=k(k+1)-l(l+1)$. Taking the span of all the eigenvectors corresponding to zero eigenvalue we get precisely (\ref{eq:KerH1m2}).
\end{proof}

\begin{lemma}
Let $m,l\in\mathbb Z_{\geq0}$ be a pair of nonnegative integers satisfying $2l\leq m$. The action of the remaining two operators (\ref{eq:H2m2}) and (\ref{eq:H3m2}) on polynomial $\mathfrak P_{m,l,l}$ is given by
\begin{subequations}
\begin{align}
X_{12}X_{23}\hat H_2^{(-2)}\mathfrak P_{m,l,l}=&(m+1)\left((m+2l+2)X_{12}X_{23}^{-1}\mathfrak P_{m,l,l}-(l+1)\mathfrak P_{m,l+1,l}-(l+1)\mathfrak P_{m,l,l+1}\right),\label{eq:H2m2PmllAction}\\[0.3em]
X_{13}X_{23}\hat H_3^{(-2)}\mathfrak P_{m,l,l}=&(m+1)\left((m+2l+2)X_{13}X_{23}^{-1}\mathfrak P_{m,l,l}+(l+1)\mathfrak P_{m,l+1,l}-(l+1)\mathfrak P_{m,l,l+1}\right).
\label{eq:H3m2PmllAction}
\end{align}
\end{subequations}
\end{lemma}
\begin{proof}
Applying differential operator (\ref{eq:H2m2}) to the right hand side of (\ref{eq:PmklDefinition}) and changing the variables as in (\ref{eq:UVXChangeOfVariables}) we get
\begin{equation}
\begin{aligned}
X_{12}X_{23}^{1-m}\hat H_2^{(-2)}\mathfrak P_{m,l,l}=&\frac12m(m+1)(U+V)\mathcal P_l(U)\mathcal P_l(V)\\
&\quad+(1+m-UV-mV^2)\mathcal P_l(U)\mathcal P_l'(V)+(1+m-mU^2-UV)\mathcal P_l'(U)\mathcal P_l(V)\\
&\quad+\frac12(U-V)(1-V^2)\mathcal P_l(U)\mathcal P_l''(V)-\frac12(U-V)(1-U^2)\mathcal P_l''(U)\mathcal P_l(V)\\
=&(m+1)\left(\frac12m(U+V)\mathcal P_l(U)\mathcal P_l(V)+(1-V^2)\mathcal P_l(U)\mathcal P_l'(V)+(1-U^2)\mathcal P_l'(U)\mathcal P_l(V)\right).
\label{eq:H2m2PmllActionStep1}
\end{aligned}
\end{equation}
Here in the second equality we have used Legendre differential equation to eliminate $\mathcal P_l''$. Finally, to eliminate the first derivatives we use the following identity
\begin{equation*}
(1-x^2)\mathcal P_l'(x)+(l+1)\left(\mathcal P_{l+1}(x)-x\mathcal P_l(x)\right)=0\qquad\textrm{for all}\quad l\in\mathbb Z_{\geq0}.
\end{equation*}
The right hand side of (\ref{eq:H2m2PmllActionStep1}) acquires the following form
\begin{equation}
\begin{aligned}
X_{12}X_{23}^{1-m}&\hat H_2^{(-2)}\mathfrak P_{m,l,l}\\[0.5em]
=&(m+1)\left((m+2l+2)\frac{U+V}2\mathcal P_l(U)\mathcal P_l(V)-(l+1)\mathcal P_{l+1}(U)\mathcal P_l(V)-(l+1)\mathcal P_l(U)\mathcal P_{l+1}(V)\right).
\end{aligned}
\label{eq:H2m2ActionPmllUV}
\end{equation}
Changing variables back to $X_{12},X_{13},X_{23}$ we conclude the proof of identity (\ref{eq:H2m2PmllAction}).

Finally, in order to prove the second identity (\ref{eq:H3m2PmllAction}) note that 
\begin{equation*}
\hat H_3^{(-2)}=\sigma_{(X_{12},X_{13})}\circ \hat H_2^{(-2)}\circ\sigma_{(X_{12},X_{13})}
\end{equation*}
where $\sigma_{(X_{12},X_{13})}$ is an order two automorphism of the ring $\mathbb C[X_{12}^{\pm1},X_{13}^{\pm1},X_{23}^{\pm1}]$ of Laurent polynomials which preserves the polynomial subring.
\end{proof}

For the remainder of Appendix \ref{sec:ProofPropositionKernelHm2} we fix lexicographic monomial ordering on Laurent monomials in $X_{12}, X_{13}, X_{23}$. Note that this is a total monomial ordering which respects multiplication, but it is not a well-order. This will be enough for comparing the leading parts of Laurent polynomials involved in our proofs.

\begin{lemma}
Let $m,l\in\mathbb Z_{\geq0}$ be a pair of nonnegative integers satisfying $m\geq 2l$. The action of differential operators $\hat H_2^{(-2)}$ and $\hat H_3^{(-2)}$ on basis element $\mathfrak P_{m,l,l}$ has the following leading term
\begin{subequations}
\begin{equation}
\mathop{\mathrm{LT}}\left(\hat H_2^{(-2)}\mathfrak P_{m,l,l}\right)=\left\{\begin{array}{ll}
\frac{((2l-1)!!)^2}{(l!)^2}(m+1)(m-2l)X_{12}^{2l}X_{23}^{m-2l-2},&m>2l,\\[0.5em]
\frac{((2l+1)!!)^2}{(l!)^2}\frac{2l^2}{4l^2-1}X_{12}^{2l-2},&m=2l>0,\\[0.5em]
\emptyset,&m=l=0,
\end{array}\right.
\label{eq:LeadingTermsH2m2Pmll}
\end{equation}
\begin{equation}
\mathop{\mathrm{LT}}\left(\hat H_3^{(-2)}\mathfrak P_{m,l,l}\right)=
\left\{\begin{array}{ll}
\frac{((2l-1)!!)^2}{(l!)^2}(m+1)(m-2l)X_{12}^{2l}X_{23}^{m-2l-2},&m>2l,\\[0.5em]
-\frac{((2l+1)!!)^2}{(l!)^2}\frac{2l^2}{4l^2-1}X_{12}^{2l-2},&m=2l>0,\\[0.5em]
\emptyset,&m=l=0.
\end{array}\right.
\label{eq:LeadingTermsH3m2Pmll}
\end{equation}
\end{subequations}
\label{lemm:LeadingTermsH2m2H3m2Pmll}
\end{lemma}
\begin{proof}

Using Rodriguez formula \cite{Askey'2005} for Legendre polynomials we get the following expression for the first two leading terms
\begin{equation}
P_l(x)=\frac1{2^l l!}\left(\frac{(2l)!}{l!}x^l-l\frac{(2l-2)!}{(l-2)!}x^{l-2}\right)+\mathbf S(x^{l-3}).
\label{eq:LegendreLeadingTerms}
\end{equation}
Here by $\mathbf S(x^{l-2})$ we denote the linear combination of monomials $x^i$ with $i<l-2$. 

With our choice of monomial ordering on $\mathbb C[X_{12}^{\pm1},X_{13}^{\pm1},X_{23}^{\pm1}]$, formula (\ref{eq:LegendreLeadingTerms}) implies that 
\begin{equation}
\begin{aligned}
&P_l\left(\frac{X_{12}\pm X_{13}}{X_{23}}\right)=\frac1{2^l}\left(\begin{array}{c}2l\\l\end{array}\right)\left(\frac{X_{12}\pm X_{13}}{X_{23}}\right)^l-\frac l{2^l}\left(\begin{array}{c}2l-2\\l\end{array}\right)\left(\frac{X_{12}\pm X_{13}}{X_{23}}\right)^{l-2}+\mathbf S(X_{12}^{l-3})\\[0.5em]
&\qquad\qquad=\frac1{2^l}\left(\begin{array}{c}2l\\l\end{array}\right)\left(X_{12}^lX_{23}^{-l}\pm lX_{12}^{l-1}X_{13}X_{23}^{-l}+\frac{l(l-1)}2X_{12}^{l-2}X_{13}^2X_{23}^{-l}-\frac{l(l-1)(l-2)}6X_{12}^{l-3}X_{13}^3X_{23}^{-l}\right)\\[0.5em]
&\qquad\qquad\quad-\frac l{2^l}\left(\begin{array}{c}2l-2\\l\end{array}\right)\left(X_{12}^{l-2}X_{23}^{2-l}\pm (l-2)X_{12}^{l-3}X_{13}X_{23}^{2-l}\right)+\mathbf S(X_{12}^{l-3}).
\end{aligned}
\label{eq:PlUVLeadingTerms}
\end{equation}
Hereinafter $\mathbf S(X_{12}^k),\; k\in\mathbb Z$ stands for linear combination of monomials $X_{12}^{i_{12}}X_{13}^{i_{13}}X_{23}^{i_{23}}$ with $i_{12}<k$.

As a corollary of (\ref{eq:PlUVLeadingTerms}) we get the following formula for the leading terms of the basis element (\ref{eq:PmklDefinition})
\begin{equation}
\begin{aligned}
\mathfrak P_{m,k,l}=&X_{23}^mP_k\left(\frac{X_{12}-X_{13}}{X_{23}}\right)P_l\left(\frac{X_{12}+X_{13}}{X_{23}}\right)=\frac1{2^{k+l}}X_{12}^{k+l}X_{23}^{m-k-l}\\[0.5em]
&\quad\times\left(\left(\begin{array}{c}2k\\k\end{array}\right)\left(1-kX_{12}^{-1}X_{13}+\frac{k(k-1)}2X_{12}^{-2}X_{13}^2-\frac{k(k-1)(k-2)}6X_{12}^{-3}X_{13}^3\right)\right.\\[0.5em]
&\qquad\quad\left.-k\left(\begin{array}{c}2k-2\\k\end{array}\right)\left(X_{12}^{-2}X_{23}^2-(k-2)X_{12}^{-3}X_{13}X_{23}^2\right)\right)\\[0.5em]
&\quad\times\left(\left(\begin{array}{c}2l\\l\end{array}\right)\left(1+lX_{12}^{-1}X_{13}+\frac{l(l-1)}2X_{12}^{-2}X_{13}^2+\frac{l(l-1)(l-2)}6X_{12}^{-3}X_{13}^3\right)\right.\\[0.5em]
&\qquad\quad\left.-l\left(\begin{array}{c}2l-2\\l\end{array}\right)\left(X_{12}^{-2}X_{23}^2+(l-2)X_{12}^{-3}X_{13}X_{23}^2\right)\right)+\mathbf S(X_{12}^{k+l-3})\\[0.5em]
=&\frac1{2^{k+l}}X_{12}^{k+l}X_{23}^{m-k-l}\left[\left(\begin{array}{c}2k\\k\end{array}\right)\left(\begin{array}{c}2l\\l\end{array}\right)\left(1+\left(l-k\right)X_{12}^{-1}X_{13}+\frac{(k-l)^2-(k+l)}2X_{12}^{-2}X_{13}^2\right)\right.\\[0.5em]
&\quad-\left(l\left(\begin{array}{c}2k\\k\end{array}\right)\left(\begin{array}{c}2l-2\\l\end{array}\right)+k\left(\begin{array}{c}2l\\l\end{array}\right)\left(\begin{array}{c}2k-2\\k\end{array}\right)\right)X_{12}^{-2}X_{23}^2\\[0.5em]
&\quad-\left(\begin{array}{c}2k\\k\end{array}\right)\left(\begin{array}{c}2l\\l\end{array}\right)\frac{(k-l)((k-l)^2-3(k+l)+2)}6X_{12}^{-3}X_{13}^3\\[0.5em]
&\quad\left.+\left(\left(\begin{array}{c}2k\\k\end{array}\right)\left(\begin{array}{c}2l-2\\l\end{array}\right)(-l)(l-k-2)+\left(\begin{array}{c}2l\\l\end{array}\right)\left(\begin{array}{c}2k-2\\k\end{array}\right)(-k)(l-k+2)\right)X_{12}^{-3}X_{13}X_{23}^2\right]\\[0.5em]
&\quad+\mathbf S(X_{12}^{k+l-3}).
\end{aligned}
\label{eq:PmklLeadingTerms}
\end{equation}
Combining (\ref{eq:H2m2PmllAction}) with (\ref{eq:PmklLeadingTerms}) and dropping all the terms which are subleading to $X_{12}^{2l}$ we get
\begin{align*}
\hat H_2^{(-2)}\mathfrak P_{m,l,l}=&(m+1)(m+2l+2)X_{12}^{2l}X_{23}^{m-2l-2}\\[0.5em]
&\qquad\times\left[\frac1{2^{2l}}\left(\begin{array}{c}2l\\l\end{array}\right)^2\left(1-lX_{12}^{-2}X_{13}^2\right)-\frac{l}{2^{2l-1}}\left(\begin{array}{c}2l\\l\end{array}\right)\left(\begin{array}{c}2l-2\\l\end{array}\right)X_{12}^{-2}X_{23}^2\right]\\[0.5em]
&-(m+1)(l+1)X_{12}^{2l}X_{23}^{m-2l-2}\\[0.5em]
&\qquad\times\left[\frac1{2^{2l+1}}\left(\begin{array}{c}2l+2\\l+1\end{array}\right)\left(\begin{array}{c}2l\\l\end{array}\right)\left(1-X_{12}^{-1}X_{13}-lX_{12}^{-2}X_{13}^2\right)\right.\\[0.5em]
&\qquad-\frac1{2^{2l+1}}\left(l\left(\begin{array}{c}2l+2\\l+1\end{array}\right)\left(\begin{array}{c}2l-2\\l\end{array}\right)+(l+1)\left(\begin{array}{c}2l\\l\end{array}\right)\left(\begin{array}{c}2l\\l+1\end{array}\right)\right)X_{12}^{-2}X_{23}^2\\[0.5em]
&\qquad+\frac1{2^{2l+1}}\left(\begin{array}{c}2l\\l\end{array}\right)\left(\begin{array}{c}2l+2\\l+1\end{array}\right)\left(1+X_{12}^{-1}X_{13}-lX_{12}^{-2}X_{13}^2\right)\\[0.5em]
&\qquad\left.-\frac1{2^{2l+1}}\left((l+1)\left(\begin{array}{c}2l\\l\end{array}\right)\left(\begin{array}{c}2l\\l+1\end{array}\right)+l\left(\begin{array}{c}2l+2\\l+1\end{array}\right)\left(\begin{array}{c}2l-2\\l\end{array}\right)\right)X_{12}^{-2}X_{23}^2\right]\\[0.5em]
&+\mathbf S(X_{12}^{2l-2})\\[0.5em]
=&(m+1)X_{12}^{2l}X_{23}^{m-2l-2}\left[(m-2l)\frac{((2l-1)!!)^2}{(l!)^2}\right.\\[0.5em]
&\qquad-\frac{l}{2^{2l}}\left(\begin{array}{c}2l\\l\end{array}\right)\left((m+2l+2)\left(\begin{array}{c}2l\\l\end{array}\right)-(l+1)\left(\begin{array}{c}2l+2\\l+1\end{array}\right)\right)X_{12}^{-2}X_{13}^2\\[0.5em]
&\qquad\left.-\frac1{2^{2l}}\left(\begin{array}{c}2l\\l\end{array}\right)\left((l+1)l\left(\begin{array}{c}2l\\l\end{array}\right)-(m+1)2^l\frac{(2l-3)!!}{(l-2)!}\right)X_{12}^{-2}X_{23}^2\right]+\mathbf S(X_{12}^{2l-2}).
\end{align*}
When $m>2l$, the first term on the right hand side is nonvanishing and coicides with the first case of (\ref{eq:LeadingTermsH2m2Pmll}). When $m=2l$, the first two terms on the right hand side vanish, while the third term simplifies exactly to the second case of (\ref{eq:LeadingTermsH2m2Pmll}). Finally, in the case $m=l=0$, the basis element is constant and thus annihinated by $\hat H_2^{(-2)}$.

Similarly, for the action of $\hat H_3^{(-2)}$ on the basis element we get
\begin{align*}
\hat H_3^{(-2)}\mathfrak P_{m,l,l}=&(m+1)(m+2l+2)X_{12}^{2l}X_{23}^{m-2l-2}\\[0.5em]
&\qquad\quad \times\left[\frac1{2^{2l}}\left(\begin{array}{c}2l\\l\end{array}\right)^2(1-l X_{12}^{-2}X_{13}^2)-\frac l{2^{2l-1}}\left(\begin{array}{c}2l\\ l\end{array}\right)\left(\begin{array}{c}2l-2\\l\end{array}\right)X_{12}^{-2}X_{23}^2\right]\\[0.5em]
&\quad +(m+1)(l+1)X_{12}^{2l+1}X_{13}^{-1}X_{23}^{m-2l-2}\\[0.5em]
&\qquad\quad\times\left[\frac1{2^{2l+1}}\left(\begin{array}{c}2l+2\\l+1\end{array}\right)\left(\begin{array}{c}2l\\l\end{array}\right)(1-X_{12}^{-1}X_{13}-lX_{12}^{-2}X_{13}^2+lX_{12}^{-3}X_{13}^3)\right.\\[0.5em]
&\qquad\quad-\frac1{2^{2l+1}}\left(l\left(\begin{array}{c}2l+2\\l+1\end{array}\right)\left(\begin{array}{c}2l-2\\l\end{array}\right)+(l+1)\left(\begin{array}{c}2l\\l\end{array}\right)\left(\begin{array}{c}2l\\l+1\end{array}\right)\right)X_{12}^{-2}X_{23}^2\\[0.5em]
&\qquad\quad+\frac1{2^{2l+1}}\left(\left(\begin{array}{c}2l+2\\l+1\end{array}\right)\left(\begin{array}{c}2l-2\\ l\end{array}\right)3l+\left(\begin{array}{c}2l\\l\end{array}\right)\left(\begin{array}{c}2l\\l+1\end{array}\right)(-l-1)\right)X_{12}^{-3}X_{13}X_{23}^2\\[0.5em]
&\qquad\quad-\frac1{2^{2l+1}}\left(\begin{array}{c}2l\\l\end{array}\right)\left(\begin{array}{c}2l+2\\l+1\end{array}\right)\left(1+X_{12}^{-1}X_{13}-lX_{12}^{-2}X_{13}^2+lX_{12}^{-3}X_{13}^3\right)\\[0.5em]
&\qquad\quad +\frac1{2^{2l+1}}\left((l+1)\left(\begin{array}{c}2l\\ l\end{array}\right)\left(\begin{array}{c}2l\\ l+1\end{array}\right)+l\left(\begin{array}{c}2l+2\\ l+1\end{array}\right)\left(\begin{array}{c}2l-2\\ l\end{array}\right)X_{12}^{-2}X_{23}^2\right)\\[0.5em]
&\qquad\quad\left. -\frac1{2^{2l+1}}\left((l+1)\left(\begin{array}{c}2l\\ l\end{array}\right)\left(\begin{array}{c}2l\\ l+1\end{array}\right)-3l\left(\begin{array}{c}2l+2\\ l+1\end{array}\right)\left(\begin{array}{c}2l-2\\ l\end{array}\right)\right)X_{12}^{-3}X_{13}X_{23}^2\right]\\[0.5em]
&\quad+\mathbf S(X_{12}^{l-2})\\[0.5em]
=&(m+1)X_{12}^{2l}X_{23}^{m-2l-2}\left[(m-2l)\frac{((2l-1)!!)^2}{(l!)^2}\right.\\[0.5em]
&\qquad-\frac1{2^{2l}}(m+1)\left(2l(m+2l+2)\left(\begin{array}{c}2l\\ l\end{array}\right)\left(\begin{array}{c}2l-2\\l\end{array}\right)+(l+1)^2\left(\begin{array}{c}2l\\ l\end{array}\right)\left(\begin{array}{c}2l\\ l+1\end{array}\right)\right.\\[0.5em]
&\qquad\quad\left.\left.-3l(l+1)\left(\begin{array}{c}2l-2\\ l\end{array}\right)\left(\begin{array}{c}2l+2\\ l+1\end{array}\right)\right)X_{12}^{-2}X_{23}^2\right]+\mathbf S(X_{12}^{2l-2}).
\end{align*}
As with the previous operator, when $m>2l$ the first therm on the right hand side is nonvanishing and coincides with the first case of (\ref{eq:LeadingTermsH3m2Pmll}). When $m=2l>0$, the first term vanishes, while the coefficient in the second term simplifies exactly to the corresponding case of (\ref{eq:LeadingTermsH3m2Pmll}). Finally, when $m=l=0$, the basis element $\mathfrak P_{0,0,0}=1$ is annihilated by $\hat H_3^{(-2)}$.
\end{proof}

\begin{lemma}
The common kernel of $\hat H_1^{(-2)}$ with either one of the two remaining operators $\hat H_2^{(-2)},\hat H_3^{(-2)}$ is a graded subspace of the polynomial ring. For all $n\in\mathbb Z_{\geq0}$, the homogeneous component of the common kernel of degree $2n+1$ is trivial
\begin{equation}
\ker\hat H_1^{(-2)}\cap\ker\hat H_j^{(-2)}\Big|_{\left(\mathbb C[X_{12},X_{13},X_{23}]\right)_{2n+1}}=\{0\},\qquad 2\leq j\leq3,
\label{eq:OddKernelH1m2Hjm2}
\end{equation}
while the homogeneous component of the common kernel of degree $2n$ is one-dimensional and given by
\begin{subequations}
\begin{align}
\ker\hat H_1^{(-2)}\cap\ker\hat H_2^{(-2)}\Big|_{\left(\mathbb C[X_{12},X_{13},X_{23}]\right)_{2n}}=&\mathbb C\phi_{2n}^{(1,2)},\quad\textrm{where}\quad \phi_{2n}^{(1,2)}:=\sum_{l=0}^n(-1)^{n-l}\frac{2l+1}{n+l+1}\left(\begin{array}{c}2n\\n-l\end{array}\right)\mathfrak P_{2n,l,l},
\label{eq:EvenKernelH1m2H2m2}\\[0.3em]
\ker\hat H_1^{(-2)}\cap\ker\hat H_3^{(-2)}\Big|_{\left(\mathbb C[X_{12},X_{13},X_{23}]\right)_{2n}}=&\mathbb C\phi_{2n}^{(1,3)},\quad\textrm{where}\quad \phi_{2n}^{(1,3)}:=\sum_{l=0}^n\frac{2l+1}{n+l+1}\left(\begin{array}{c}2n\\n-l\end{array}\right)\mathfrak P_{2n,l,l}.
\label{eq:EvenKernelH1m2H3m2}
\end{align}
\label{eq:EvenKernelH1m2Hjm2}
\end{subequations}
\end{lemma}
\begin{proof}
Let $m\in\mathbb Z_{\geq0}$ be a nonnegative integer. From (\ref{eq:KerH1m2}) we know that
\begin{equation}
\left\{\mathfrak P_{m,l,l}:0\leq l\leq\left\lfloor\frac m2\right\rfloor\right\}
\label{eq:KernelH1m2BasisIntersectionTwo}
\end{equation}
forms a basis on the kernel of operator $\hat H_1^{(-2)}$.

When $m\equiv1\bmod2$, it follows by (\ref{eq:LeadingTermsH2m2Pmll}) that the leading terms of the action of $\hat H_2^{(-2)}$ on basis elements (\ref{eq:KernelH1m2BasisIntersectionTwo}) are all distinct. Hence, the homogeneous component of the common kernel of $\hat H_1^{(-2)}$ and $\hat H_2^{(-2)}$ of degree $m$ must be trivial.

Now, assume $m=2n$ for some $n\in\mathbb N$ and let 
\begin{equation}
\psi=\sum_{l=0}^{n}c_l\mathfrak P_{2n,l,l}\quad\in\quad \ker\hat H_1^{(-2)}\cap\ker\hat H_2^{(-2)}\Big|_{\left(\mathbb C[X_{12},X_{13},X_{23}]\right)_{2n}}
\label{eq:PsiElementKernelH1m2H2m2}
\end{equation}
be a nontrivial element of the homonegeous component of the common kernel of degree $2n$. Again, by (\ref{eq:LeadingTermsH2m2Pmll}) it follows that necessarily $c_l\neq 0$. Hence the homogeneous component of the common kernel is at most one-dimensional.

It remains to show that $\phi_{2n}^{(1,2)}$ belongs to the kernel. By (\ref{eq:KerH1m2}) we know that $\phi_{2n}^{(1,2)}$ is annihilated by $\hat H_1^{(-2)}$. As for the action of $\hat H_2^{(-2)}$, consider an arbitrary element of the form (\ref{eq:PsiElementKernelH1m2H2m2}),  we have
\begin{equation*}
\begin{aligned}
\frac{X_{12}X_{23}^{m-1}}{m+1}&\hat H_2^{(-2)}\psi=\frac{X_{12}X_{23}^{m-1}}{m+1}\sum_{l=0}^nc_{l}\hat H_2^{(-2)}\mathfrak P_{2n,l,l}\\[0.5em]
\stackrel{(\ref{eq:H2m2ActionPmllUV})}{=}&\sum_{l=0}^n c_l \left((m+2l+2)\frac{U+V}2\mathcal P_l(U)\mathcal P_l(V)-(l+1)\mathcal P_l(U)\mathcal P_l(V)-(l+1)\mathcal P_l(U)\mathcal P_{l+1}(V)\right)\\[0.5em]
=\;\,&\sum_{l=0}^{n-1} \left( c_{l+1} \dfrac{(n + l + 2)(l+1)}{2l + 3} + c_{l} \dfrac{(n-l)(l+1)}{2l + 1} \right) \big( \mathcal P_l(U)\mathcal P_{l+1}(V) + \mathcal P_{l+1}(U)\mathcal P_{l}(V) \big)  
\end{aligned}
\end{equation*}
The right hand side of the above expression vanishes for $c_l$ as in (\ref{eq:EvenKernelH1m2H2m2}), namely when we specialize to
\begin{equation*}
c_l=(-1)^{n-l}\frac{2l+1}{n+l+1}\left(\begin{array}{c}2n\\ n-l\end{array}\right).
\end{equation*}

The proof of the remaining part (\ref{eq:EvenKernelH1m2H3m2}) is completely analogous to the argument presented above and thus omitted for the sake of brevity.
\end{proof}

\begin{customprop}{\ref{prop:KernelHm2}}
Let $p\in\mathbb C[X_{12},X_{13},X_{23}]$ be a polynomial annihilated by all three of the operators (\ref{eq:Hm2}), then necessarily it is a constant polynomial $p\in\mathbb C$. In other words,
\begin{equation}
\ker\hat H_1^{(-2)}\cap\ker\hat H_2^{(-2)}\cap\ker\hat H_3^{(-2)}=\mathbb C.
\label{eq:CommonKernelHm2}
\end{equation}
\end{customprop}
\begin{proof}
Recall that the three operators $\hat H_i^{(-2)},\;1\leq i\leq3$ are all homogeneous of degree $-2$. As a result, their common kernel is a graded subspace of the polynomial ring. The constant polynomial is annihilated by all three operators, so the homogeneous component of degree zero of the common kernel is $\mathbb C$. To prove (\ref{eq:CommonKernelHm2}) we have thus left to show that all higher homogeneous components of the common kernel are all trivial.

Let $m\in\mathbb N$ be a natural number. When $m\equiv 1\bmod 2$, by (\ref{eq:OddKernelH1m2Hjm2}) the corresponding homogeneous component must be trivial. Next, when $m=2n$ for some $n\in\mathbb N$, it follows by (\ref{eq:EvenKernelH1m2Hjm2}) that every homogeneous element $\psi$ of degree $2n$ in the common kernel must be propostional simultaneously to $\phi_{2n}^{(1,2)}$ and $\phi_{2n}^{(1,3)}$. The latter implies that $\psi=0$ and that completes the proof.
\end{proof}

\section*{Acknowledgements}
W.Y. is supported by National Key R\&D Program of China (2022ZD0117000)

\bibliographystyle{alpha}
\bibliography{references}

\end{document}